\documentclass[a4paper,12pt]{amsart}
\usepackage[top=3cm,bottom=3cm,outer=3cm,inner=3cm,marginpar=2.45cm]{geometry}
\usepackage{latexsym,amsmath}
\usepackage{amsthm,amssymb}

\newtheorem{theorem}{Theorem}[section]
\newtheorem{lemma}[theorem]{Lemma}
\newtheorem{corollary}[theorem]{Corollary}
\newtheorem{proposition}[theorem]{Proposition}

\theoremstyle{definition}
\newtheorem{definition}[theorem]{Definition}

\newtheorem{remark}[theorem]{Remark}

\numberwithin{equation}{section}

\def\RR{\mathbb{R}} % the real number R
\def\CC{\mathbb{C}} % Complex C
\def\DD{\mathbf{D}} % unit disc D
\def\NN{\mathbb{N}} % Natural number N
\def\im{\sqrt{-1}} % the imaginary number i
\def\ddbar{\partial\bar\partial} %\ve
\def\vol{\mathrm{Vol}} %volume
\def\ve{\varepsilon}
\def\vp{\varphi}
\def\beq{\begin{equation*}}
\def\eeq{\end{equation*}}

\newcommand{\paren}[1]{\left(#1\right)}
\newcommand{\bparen}[1]{\left[#1\right]}
\newcommand{\pd}[2]{\frac{\partial#1}{\partial#2}}
\newcommand{\abs}[1]{\left\vert#1\right\vert}
\newcommand{\norm}[1]{\left\|#1\right\|}

\newcommand{\ov}[1]{\overline{#1}}
\newcommand{\inner}[1]{\left\langle{#1}\right\rangle}
\newcommand{\ind}[3]{_{#1\phantom{#2}#3}^{\phantom{#1}#2}}

\newcommand{\ip}[1]{\mathrm{Im}\;s}

\begin{document}

%------------------------------------------------------------------

\title[Direct images of Fiberwise Ricci-flat metrics]{Positivity of direct images of fiberwise Ricci-flat metrics on Calabi-Yau fibrations}

\author{Matthias Braun}
\address{Fachbereich Mathematik und Informatik, Philipps-Universit\"at Marburg, Lahnberge, Hans-Meerwein-Stra{\ss}e, D-35032 Marburg, Germany}
\email{braunm@mathematik.uni-marburg.de}

\author{Young-Jun Choi}

\address{Department of Mathematics, Pusan National University, 2, Busandaehak-ro 63beon-gil, Geumjeong-gu, Busan, 46241, Korea}

\email{youngjun.choi@pusan.ac.kr}

\author{Georg Schumacher}

\address{Fachbereich Mathematik und Informatik, Philipps-Universit\"at Marburg, Lahnberge, Hans-Meerwein-Stra\ss e, D-35032 Marburg, Germany}
\email{schumac@mathematik.uni-marburg.de}

\subjclass[2010]{32Q25, 32Q20, 32G05, 32W20}

\keywords{Calabi-Yau manifold, Ricci-flat metric, K\"ahler-Einstein metric, a family of Calabi-Yau manifolds, variation}

\begin{abstract}
Let $X$ be a K\"ahler manifold which is fibered over a complex manifold $Y$ such that every fiber is a Calabi-Yau manifold. 
Let $\omega$ be a fixed K\"ahler form on $X$. 
By Yau's theorem, there exists a unique Ricci-flat K\"ahler form $\rho\vert_{X_y}$ for each fiber, which is cohomologous to $\omega\vert_{X_y}$. 
This family of Ricci-flat K\"ahler forms $\rho\vert_{X_y}$ induces a smooth $(1,1)$-form $\rho$ on $X$ with a normalization condition. 
In this paper, we prove that the direct image of $\rho^{n+1}$ is positive on the base $Y$.
We also discuss several byproducts, among them the local triviality of families of Calabi-Yau manifolds.
\end{abstract}

\maketitle

\section{Introduction}

Let $p:X\rightarrow Y$ be a proper surjective holomorphic mapping between complex manifolds $X$ and $Y$ whose differential has maximal rank everywhere such that every fiber $X_y:=p^{-1}(y)$ is a compact K\"ahler manifold. This is called a \emph{smooth family of compact K\"ahler manifolds} or a \emph{compact K\"ahler fibration}. If every fiber $X_y$ is a Calabi-Yau manifold, i.e., a compact K\"ahler manifold whose canonical line bundle $K_{X_y}$ is trivial, then the family is called a \emph{smooth family of Calabi-Yau manifolds} or a \emph{Calabi-Yau fibration}. 

If $(X,\omega)$ is a K\"ahler mainfiold, then a celebrated theorem due to Calabi and Yau implies that on each fiber $X_y$, there exists a unique Ricci-flat metric $\omega_{KE,y}$ in the cohomology class $[\omega\vert_{X_y}]$. 
This family of Ricci-flat metrics induces a fiberwise Ricci-flat metric on the total space $X$.
\medskip

The main theorem of this paper is the following:

\begin{theorem} \label{T:main_theorem}
Let $p:X\rightarrow Y$ be a smooth family of Calabi-Yau manifolds.
Suppose that $X$ is a K\"ahler manifold equipped with a K\"ahler form $\omega$.
Let $\omega_{KE,y}$ be the unique Ricci-flat form in the cohomology class $[\omega\vert_{X_y}]$. Then there exists a unique smooth function $\vp\in C^\infty(X)$ which satisfies the following properties:
\begin{itemize}
\item[(\romannumeral1)] $\int_{X_y}\vp(\omega_{KE,y})^n=0$ for every $y\in Y$,
\item[(\romannumeral2)] $\omega+dd^c\vp\vert_{X_y}$ is a Ricci-flat K\"ahler form on $X_y$ for every $y\in Y$ and
\item[(\romannumeral3)] $p_*\paren{\omega+dd^c\vp}^{n+1}$ is a positive $(1,1)$-form on $Y$.
\end{itemize}
\end{theorem}

Here $d^c$ means the real operator defined by $d^c=\frac{\im}{2}\paren{\partial-\bar\partial}$.
Then we have $dd^c=\im\ddbar$.
We call the $(1,1)$-form $\rho:=\omega+dd^c\vp$ which satisfies the property $(\romannumeral2)$ a \emph{fiberwise Ricci-flat metric} or a \emph{semi Ricci-flat K\"ahler form} on a Calabi-Yau fibration $p:X\rightarrow Y$.
Note that a real $(1,1)$-form on $X$ satisfying $(\romannumeral2)$ is not uniquely determined.
With the normalization condition $(\romannumeral1)$, a fiberwise Ricci-flat metric is uniquely determined. From now on, the fiberwise Ricci-flat metric on a Calabi-Yau fibration means the real $(1,1)$-form which satisfies $(\romannumeral1)$ and $(\romannumeral2)$.
\medskip

For a family of canonically polarized compact K\"ahler manifolds, we have a fiberwise K\"ahler-Einstein metric by the similar way.
The positivity of the fiberwise K\"ahler-Einstein metric on a family of compact K\"ahler manifolds was first studied by Schumacher.
In his paper \cite{Schumacher}, he have proved that the fiberwise K\"ahler-Einstein metric on a family of canonically polarized compact K\"ahler manifolds is semi-positive.
Moreover he have also proved that it is strictly positive if the family is effectively parametrized.
This is equivalent to the semi-positivity or positivity of the relative canonical line bundle of the family, respectively.
P\v aun have shown that if the relative adjoint line bundle is positive on each fiber, then it is semi-positive on the total space by generalizing the method of Schumacher (\cite{Paun2}).
Guenancia also have proved the semi-positivity of the fiberwise conic singular K\"ahler-Einstein metric  (\cite{Guenancia}).
In case of a family of complete K\"ahler manifolds, Choi have proved that the fiberwise K\"ahler-Einstein metric on a family of bounded pseudoconvex domains is semi-positive or positive if the total space is pseudoconvex or strongly pseudoconvex, respectively (\cite{Choi1, Choi2}).
\medskip

The proof of Schumacher's theorem starts with the following identity from \cite{Semmes}: For a real $(1,1)$-form $\tau$ on $X$,
\begin{equation}\label{E:Semmes0}
\tau^{n+1}
=
c(\tau)\tau^n\wedge\im ds\wedge d\bar s
\end{equation}
where $\tau^n$ is the $n$-fold exterior power divided by $n!$. Here $c(\tau)$ is called a \emph{geodesic curvature} of $\tau$. (For the detail, see Section \ref{SS:horizontal_lift}.) Now suppose that $\tau$ is positive-definite on each fiber $X_y$. Then \eqref{E:Semmes0} says that $\tau$ is semi-positive or positive if and only if $c(\tau)\ge0$ or $c(\tau)>0$, respectively. Schumacher have proved that the geodesic curvature of the fiberwise K\"ahler-Einstein metric on a family of canonically polarized compact K\"ahler manifolds satisfies a certain second order linear elliptic partial differential equation. This PDE gives a lower bound of the geodesic curvature by the maximum principle or an lower bound estimate on the heat kernel.

However, in case of a Calabi-Yau fibration, the PDE which the geodesic curvature of fiberwise Ricci-flat metric satisfies is of a different type from the previous one.
(See Section \ref{S:fiberwiseRFf}.)
In particular, it does not give a lower bound of the geodesic curvature. This is why Schumacher's method does not give positivity or semi-positivity of the fiberwise Ricc-flat metric.
But the approximation procedure of complex Monge-Ampere equations, it is possible to obtain a lower bound of the direct image of the fiberwise Ricci-flat metric (see Section \ref{S:proof}.)
This is the main contribution of this paper.

This difference of the PDEs, which the fiberwise K\"ahler-Einstein metric on a family of canonically polarized manifolds and the fiberwise Ricc-flat metric on a family of Calabi-Yau manifold satisfy, arises from the difference of complex Monge-Amp\`ere equations which give the K\"ahler-Einstein metrics. More precisely, the complex Monge-Amp\`ere equation of type:
\begin{equation}\label{E:1}
\paren{\omega+dd^c\vp}^n
=
e^{\lambda\vp+f}\omega^n,
\end{equation}
for some constant $\lambda>0$ and some suitable smooth function $f$, gives the K\"aher-Einstein metric on a canonically polarized compact K\"aher manifold. On the other hand, the complex Monge-Amp\`ere equation of type:
\begin{equation}\label{E:0}
\paren{\omega+dd^c\vp}^n
=
e^{\tilde{f}}\omega^n
\end{equation}
for some suitable smooth function $\tilde f$, gives the K\"ahler-Einstein (in this case Ricci-flat) metric on a Calabi-Yau manifold. It is remarkable to note that if $f$ and $\tilde f$ coincide,  then \eqref{E:1} converges to \eqref{E:0} as $\lambda\rightarrow0$.
Then by the a priori estimate for complex Monge-Amp\`ere equation, it is well known that the solutions $\vp_\lambda$ of \eqref{E:1} converges to the solution  of \eqref{E:0} (see Section \ref{S:approximation}).
This is the key obervation of the proof of approximation procedures which we already mentioned.

\section{Preliminaries}

Let $p:X^{n+d}\rightarrow Y^d$ be a smooth family of K\"ahler manifolds. Taking a local coordinate $(s^1,\dots,s^d)$ of $Y$ and a local coordinate $(z^1,\dots,z^n)$ of a fiber of $p$, $(z^1,\dots,z^n,s^1,\dots,s^d)$ forms a local coordinate of $X$ such that under this coordinate, the holomorphic mapping $p$ is locally given by
\begin{equation*}
p(z^1,\dots,z^n,s^1,\dots,s^d)
=
(s^1,\dots,s^d).
\end{equation*}
We call this an \emph{admissible coordinate of $p$}.

Throughout this paper we use  small Greek letters, $\alpha,\beta,\dots=1,\dots,n$ for indices on $z=(z^1,\dots,z^n)$ and small roman letters, $i,j,\dots=1,\dots,d$ for indices on $s=(s^1,\dots,s^d)$ unless otherwise specified. For a properly differentiable function $f$ on $X$, we denote by
\begin{equation} \label{E:convention}
f_\alpha
    =\pd{f}{z^\alpha},\;\;
f_{\bar\beta}
    =\pd{f}{z^{\bar\beta}},\;\;
   \;\;\text{and}\;\;
f_{i}
    =\pd{f}{s^i},\;\;
f_{\bar{j}}
    =\pd{f}{s^{\bar{j}}},
\end{equation}
where $z^{\bar\beta}$ and $s^{\bar{j}}$ mean $\overline{z^\beta}$ and $\overline{s^j}$, respectively. In case $d=1$, we denote by
\begin{equation*}
f_{s}
	=\pd{f}{s}\;\;
\text{and}\;\;
f_{\bar{s}}
	=\pd{f}{\bar{s}}.
\end{equation*}
If there is no confusion, we always use the Einstein convention. For simplicity we denote by $v_i:=\partial/\partial{s^i}$. If $d=1$, then we denote by $v:=\partial/\partial{s}$.

\subsection{Horizontal lifts and geodesic curvatures}\label{SS:horizontal_lift}

For a complex manifold $M$, we denote by $T'M$ the complex tangent bundle of type $(1,0)$.
\begin{definition} \label{D:lift&curvature}
Let $V\in T'Y$ and $\tau$ be a real $(1,1)$-form on $X$. Suppose that $\tau$ is positive definite on each fiber $X_y$.
\begin{itemize}
\item[1.] A vector field $V_\tau$ of type $(1,0)$ is  called a \emph{horizontal lift} of $V$ if $V_\tau$ satisfies the following:
\begin{itemize}
\item [(\romannumeral1)]$\inner{V_\tau,W}_\tau=0$ for all $W\in{T'X_y}$,
\item [(\romannumeral2)]$d\pi(V_\tau)=V$.
\end{itemize}
\item[2.] The \emph{geodesic curvature} $c(\tau)(V)$ of $\tau$ along $V$ is defined by the norm of $V_\tau$ with respect to the sesquilinear form $\inner{\cdot,\cdot}_\tau$ induced by $\tau$, namely,
\begin{equation*}
c(\tau)(V)=\inner{V_\tau,V_\tau}_\tau.
\end{equation*}
\end{itemize}
\end{definition}

\begin{remark} \label{R:horizontal_lift}
Let $(z^1,\dots,z^n,s^1,\dots,s^d)$ be an admissible coordinate of $p$. Then we can write $\tau$ as follows:
\begin{equation*}
\tau
=
\im\paren{\tau_{i\bar{j}}ds^i\wedge{ds}^{\bar{j}}
+\tau_{i\bar\beta}ds^i\wedge{dz}^{\bar\beta}
+\tau_{\alpha\bar{j}}dz^\alpha\wedge{ds}^{\bar{j}}
+\tau_{\alpha\bar\beta}dz^\alpha\wedge{dz}^{\bar\beta}
}.
\end{equation*}
Since $\tau$ is positive-definite on each fiber $X_y$, the matrix $(\tau_{\alpha\bar\beta})$ is invertible. We denote by $(\tau^{\bar\beta\alpha})$ the inverse matrix. Then it is easy to see that the horizontal lift of $\partial/\partial{s^i}$ is given as follows.
\begin{equation*}
\paren{\pd{}{s^i}}_\tau
=
\pd{}{s^i}-\tau_{i\bar\beta}\tau^{\bar\beta\alpha}\pd{}{z^\alpha},
\end{equation*}
in particular, any horizontal lift with respect to $\tau$ is uniquely determined.
On the other hand, the geodesic curvature $c(\tau)(v_i)$ is computed as follows:
\begin{equation*}
\begin{aligned}
c(\tau)(v_i)
=
\inner{(v_i)_\tau,(v_i)_\tau}_\tau
=
\tau_{i\bar i}
-\tau_{i\bar\beta}\tau^{\bar\beta\alpha}\tau_{\alpha\bar i}.
\end{aligned}
\end{equation*}
\end{remark}

\begin{remark} \label{R:metric_K_X/Y}
The real $(1,1)$-form $\tau$ in Definition \ref{D:lift&curvature} induces a hermitian metric on the relative canonical line bundle $K_{X/Y}$ as follows:

Let $(z^1,\dots,z^n,s^1,\dots,s^d)$ be an admissible coordinate in $X$.
Since $\tau$ is positive-definite on each fiber, $(\tau_{\alpha\bar\beta})$ is positive-definite. 
Hence $\sum\tau_{\alpha\bar\beta}(z,s)dz^\alpha\wedge dz^{\bar\beta}$ gives a K\"ahler metric on each fiber $X_s$. 
It follows that 
\begin{equation}\label{E:metric_on_RCLB}
\det(\tau_{\alpha\bar\beta}(z,s))^{-1}
\end{equation}
gives a hermitian metric on the relative line bundle $K_{X/Y}$. We denote this metric by $h^\tau_{X/Y}$. 
The curvature form $\Theta_\tau=\Theta_{h^\tau_{X/Y}}(K_{X/Y})$ of $h^\tau_{X/Y}$ is given by
\begin{equation*}
\Theta_{h^\tau_{X/Y}}(K_{X/Y})
=
dd^c\log\det(\tau_{\alpha\bar\beta}(z,s)).
\end{equation*}
It is obvious that the cuvature is also written as follows:
\begin{equation*}
\Theta_{h^\tau_{X/Y}}(K_{X/Y})
=
dd^c\log\det\paren{\tau^n\wedge dV_s}.
\end{equation*}
\end{remark}

Suppose that $Y$ is $1$-dimensional. 
Then it is well known (cf, see \cite{Semmes}) that 
\begin{equation}\label{E:Semmes}
\tau^{n+1}
=
c(\tau)\cdot \tau^n\wedge\im{d}s\wedge{d}\bar{s}.
\end{equation}
It follows that if $c(\tau)>0 \; (\ge0)$, then $\tau$ is a positive (semi-positive) real $(1,1)$-form as $\tau$ is positive definite when restricted to $X_s$. On the other hand, \eqref{E:Semmes} says that
\begin{equation*}
p_*\tau^{n+1}
=
\int_{X_s}\tau^{n+1}
=
\int_{X_s}c(\tau)\cdot \tau^n\wedge\im{d}s\wedge{d}\bar{s}.
\end{equation*}
Hence $p_*\tau^{n+1}$ is positive or semi-positive if and only if $\int_{X_s}c(\tau)\tau^n$ is positive or nonnegative, respectively.
For later use, we introduce the following lemma.

\begin{lemma}\label{L:contraction}
The following identity holds:
\begin{equation*}
i_{v_\tau} \tau
=
\im c(\tau)d\bar s.
\end{equation*}
\end{lemma}

\begin{proof}
The computation is quite straightforward.
\begin{align*}
i_{v_\tau} \tau
&=
\im\paren{
\tau_{s\bar s}d\bar s
+
\tau_{s\bar\beta}dz^{\bar\beta}
-
\tau_{s\bar\beta}\tau^{\bar\beta\alpha}\tau_{\alpha\bar s}d\bar s
-
\tau_{s\bar\delta}\tau^{\bar\delta\alpha}\tau_{\alpha\bar\beta}dz^{\bar\beta}
}
\\
&=
\im\paren{
\tau_{s\bar s}d\bar s
-
\tau_{s\bar\beta}\tau^{\bar\beta\alpha}\tau_{\alpha\bar s}d\bar s
}
=
\im c(\tau)d\bar s.
\end{align*}
This completes the proof.
\end{proof}

\subsection{Kodaira-Spencer classes and Direct image bundles}

Let $p:X\rightarrow Y$ be a smooth family of compact K\"ahler manifolds. We denote the Kodaira-Spencer map for the family $p:X\rightarrow Y$ at a given point $y\in Y$ by
\begin{equation*}
K_y:T'_y Y\rightarrow
H^1(X_y,T'X_y).
\end{equation*}
The Kodaira-Spencer map is induced by the edge homomorphism for the short exact sequence
\begin{equation*}
0\rightarrow
T'_{X/Y}
\rightarrow
T'X
\rightarrow
p^*T'Y
\rightarrow0.
\end{equation*}
If $V\in T'_y Y$ is a tangent vector, and if $V+b^\alpha\pd{}{z^\alpha}$ is any smooth lifting of $V$ along $X_y$, then 
\begin{equation*}
\bar\partial
\paren{
	V+b^\alpha\pd{}{z^\alpha}
}
=
\pd{b^\alpha}{z^{\bar\beta}}
\pd{}{z^\alpha}
\otimes
dz^{\bar\beta}
\end{equation*}
is a $\bar\partial$-closed form on $X$, which represents $K_y(V)$, i.e.,
\begin{equation*}
K_y(V)
=
\left[
\pd{b^\alpha}{z^{\bar\beta}}
\pd{}{z^\alpha}
\otimes
dz^{\bar\beta}
\right]
\in H^{0,1}(X_y,T'X_y).
\end{equation*}
This cohomology class $K_y(V)$ is called the \emph{Kodaira-Spencer class} of $V$. The celebrated theorem of Kodaira and Spencer says that if the Kodaira-Spencer class vanishes locally, then the family is locally trivial (\cite{Kodaira:Spencer}, see also \cite{Kodaira}).
\medskip

The direct image sheaf $E:=p_*(K_{X/Y})$ of $K_{X/Y}$ is defined by the sheaf over $Y$ whose fiber $E_y$ is given by
\begin{equation*}
E_y=H^0(X_y,K_{X_y}).
\end{equation*}
It is remarkable to note that this sheaf is indeed a holomorphic vector bundle by the Ohsawa-Takegoshi extension theorem (for more details, see Section 4 in \cite{Berndtsson1}). $E$ is a hermitian vector bundle with $L^2$ metric defined by following: For $u_y, v_y\in E_y$, define $\inner{u_y,v_y}$ by
\begin{equation*}
\inner{u_y,v_y}_y^2
=
\int_{X_y}c_n u_y\wedge \overline{v_y}
\end{equation*}
where $c_n=(\im)^{n^2}$ chosen to make the form positive. The Kodaira-Spencer class acts on $u_y\in E_y$ as follows: Let $k_y(V)$ be any representative of $K_y(V)$, i.e., $T'X_y$-valued $(0,1)$-form in $K_y(V)$, which locally decomposes as 
\begin{equation*}
k_y=\zeta\otimes w
\end{equation*}
where $\zeta$ is a $(0,1)$-form and $w$ is a vector field of type $(1,0)$. Then $k_y(V)$ acts on $u_y$ by
\begin{equation*}
k_y(V)\cdot u_y=\zeta\wedge(i_w(u_y)),
\end{equation*}
where $i_w$ is the contraction. This gives a globally defined $\bar\partial$-closed form of type $(n-1,1)$ and
\begin{equation*}
K_y(V)\cdot u_y
:=
\bparen{k_y(V)\cdot u_y
}
\in H^{n-1,1}(X_y).
\end{equation*}
The following theorem due to Griffiths says the curvature of $E$ is computed in terms of Kodaira-Spencer classes (\cite{Griffiths}, see also \cite{Berndtsson2}).

\begin{theorem}\label{T:Griffiths}
Let $\Theta(E)$ be the curvature of $E$ with $L^2$-metric. Then for $V\in T_y'Y$,
\begin{equation}\label{E:Griffiths}
\inner{\Theta_{V\bar V}(E)u,u}=\norm{K_y(V)\cdot u}^2,
\end{equation}
where $\norm{K_y(V)\cdot u}$ is the norm of its unique harmonic representative. 
In particular, it does not depend on the choice of K\"ahler metric.
\end{theorem}

\section{Approximations of complex Monge-Amp\`ere equations}\label{S:approximation}

In this section, we discuss approximations of a solution of complex Monge-Amp\`ere equation \eqref{E:0} in terms of the solutions of \eqref{E:1}. First we consider the approximation on a single compact K\"ahler manifold. After that, we apply the approximation procedure to a family of complex Monge-Amp\`ere equations. First, we recall the existence and uniqueness theorem of complex Monge-Amp\`ere equations due to Aubin and Yau.

Let $(X,\omega)$ be a compact K\"ahler manifold. Let $f$ be a smooth function on $X$. The complex Monge-Amp\`ere equation is given as follows:
\begin{equation} \label{E:CMAE}
\begin{aligned}
\paren{\omega+dd^c\vp}^n
&=
e^{\lambda\vp+f}\omega^n,
\\
\omega+dd^c\vp
&>0.
\end{aligned}
\end{equation}
This fully nonlinear complex partial differential equation was first raised by E. Calabi. 
If $\lambda\ge0$, then we can solve the PDE.

\begin{theorem} \label{T:AY}
The following holds:
\begin{itemize}
\item [1.] (Aubin/Yau \cite{Aubin1, Yau}) If $\lambda>0$, then there exists a unique smooth function $\vp$ satisfying \eqref{E:CMAE} for every smooth function $f\in C^\infty(X)$. 
\item [2.] (Yau \cite{Yau}) If $\lambda=0$, then there exists a smooth function $\vp$ satisfying \eqref{E:CMAE} for $f\in C^\infty(X)$ such that $\int_X e^f\omega^n=\int_X \omega^n$. Moreover, the solution is unique up to the addition of constants.
\end{itemize}
\end{theorem}

\subsection{Approximation on a compact K\"ahler manifold}
\label{SS:approximation1}

Let $(M,\omega)$ be a compact K\"ahler manifold and $f$ be a smooth function on $M$ satisfying
$\int_M e^f\omega^n=\int_M\omega^n$.
Consider the following complex Monge-Amp\`ere equation:
\begin{equation} \label{E:CMAE0}
\begin{aligned}
\paren{\omega+dd^c\varphi}^n &= e^f\omega^n, \\
\omega+dd^c&\varphi>0.
\end{aligned}
\end{equation}
Theorem \ref{T:AY} implies that there exists a solution which is unique up to addition of constants. 

Let $\{f_\ve\}$ be a sequence of smooth functions in $M$ which converges to $f$ as $\ve$ goes to $0$ in $C^{k,\alpha}(M)$-topology for any $k\in\NN$ and $\alpha\in(0,1)$. We want to approximate a solution of \eqref{E:CMAE0} by the solutions $\vp_\ve$ of the following complex Monge-Amp\`ere equations:
\begin{equation} \label{E:CMAE1}
\begin{aligned}
\paren{\omega+dd^c\vp_\ve}^n &= 
e^{\ve\vp_\ve+f_\ve}\omega^n \\
\omega+dd^c&\vp_\ve>0,
\end{aligned}
\end{equation}
as $\ve\rightarrow0$.  Note that if $\ve\rightarrow0$, then Equation \eqref{E:CMAE1} converges to Equation \eqref{E:CMAE0}. 

The convention all over this paper is that we will use the same letter ``$C$'' to denote a generic constant, which may change from one line to another, but it is independent of the pertinent parameters involved (especially $\ve$).

\begin{proposition} \label{P:approximation1}
For each $\ve$ with $0<\ve\le1$, let $\vp_\ve$ be the solution of \eqref{E:CMAE1}. Then for any $k\in\NN$ and $\alpha\in(0,1)$, there exists a constant $C>0$ which depend only on $k$, $\alpha$, the geometry of $(M,\omega)$ and the function $f$ such that 
\begin{equation*}
\norm{\vp_\ve}_{C^{k,\alpha}(M)}<C.
\end{equation*}
In particular, $\{\vp_\ve\}$ is a relatively compact subset of $C^{k,\alpha}(M)$ for any positive integer $k$ and $\alpha\in(0,1)$.
\end{proposition}

\begin{proof}
We may assume that $\mathrm{Vol}(M)=\int_X\omega^n=1$.
The first step is obtaining a uniform upper bound for $\varphi_\varepsilon$. For each $\varepsilon>0$, the solution $\varphi_\varepsilon$ of \eqref{E:CMAE1} satisfies that 
\begin{equation*}
1=\int_M\paren{\omega+dd^c\vp_\ve}^n
=
\int_M e^{\ve\vp_\ve} e^{f_\ve}\omega^n
\end{equation*}
Then Jensen inequality implies that
\begin{equation*}
1\ge
\exp\paren{\int_M\ve\vp_\ve e^{f_\ve}\omega^n},
\end{equation*}
it is equivalent to
\begin{equation*}
\int_M\vp_\ve e^{f_\ve}\omega^n
\le
0.
\end{equation*}
Note that $f_\ve$ converges to $f$ as $\ve\rightarrow0$. 
The Hartogs lemma for quasi-plurisubharmonic functions implies that
\begin{equation}\label{E:upper}
\sup_M\vp_\ve
<
C,
\end{equation}
where $C$ is a constant which depends only on the geometry of $(M,\omega)$ and $f$ (\cite{Guedj_Zerihai}). Here we recall the simple version of Ko{\l}odziej's uniform estimates (for the general theorem, see \cite{Kolodziej1, Kolodziej2}).

\begin{theorem}\label{T:uniform}
Let $(M,\omega)$ be a compact K\"ahler manifold. Assume that $\varphi$ satisfies the following complex Monge-Amp\`ere equation:
\begin{align*}
\paren{\omega+dd^c\varphi}^n 
&=
F\omega^n, \\
\omega+dd^c\varphi
&>
0.
\end{align*}
Then
\begin{equation*}
\norm\vp_{C^0(M)}\le
C
\end{equation*}
where $C>0$ depends only on $(M,\omega)$ and on an upper bound for $\norm{F}_p$ for some $1<p\le\infty$.
\end{theorem}

If we set $F=e^{\ve\vp_\ve+f_\ve}$, then $\abs{F}<C$ for some $C>0$ by \eqref{E:upper}. Then it follows from Theorem \ref{T:uniform} that 
\begin{equation}\label{E:uniform}
\norm{\vp_\ve}_{C^0(M)}<C
\end{equation}
for some $C>0$ which depends only on $M$ and the function $f$.
\medskip

The second step is obtaining the Laplacian estimates. We recall the following theorem in \cite{DiNezza_Lu}, which is essentially due to M. P\v{a}un  (\cite{Paun1}, cf. see \cite{Siu}).
\begin{theorem}\label{T:Laplacian}
Let $\psi^+$ and $\psi^-$ be smooth quasi-plurisubharmonic functions on $M$. Let $\varphi\in~C^\infty(M)$ be such that $\sup_M\varphi=0$ and
\begin{equation*}
(\omega+dd^c\varphi)^n
=
e^{\psi^+-\psi^-}\omega^n.
\end{equation*}
Assume given a constant $C>0$ such that
\begin{equation*}
dd^c\psi^\pm\ge-C\omega,
\;\;\;
\sup_M\psi^+\le C.
\end{equation*}
Assume also that the holomorphic bisectional curvature of $\omega$ is bounded from below by $-C$. Then there exists $A>0$ depending on $C$ and $\int_Me^{-2(4C+1)\varphi}\omega^n$ such that
\begin{equation*}
0\le
n+\Delta_\omega\varphi
\le
Ae^{-2\psi^-}.
\end{equation*}
\end{theorem}

We take $\psi^+=\ve\vp_\ve+f_\ve$ and $\psi^-=0$. Since $f_\ve$ converges to $f$ as $\ve\rightarrow0$ and every $\vp_\ve$ satisfies that
\begin{equation*}
dd^c\vp_\ve > -\omega,
\end{equation*}
it follows from \eqref{E:uniform} that $\psi^+$ satisfies the hypothesis of Theorem~\ref{T:Laplacian}. Note that $\{\vp_\ve\}_{0<\ve\le1}$ is a relatively compact subset of $L^1(X,\omega)$. This implies the Laplacian estimates for $\vp_\ve$:
\begin{equation*}
\abs{\Delta_\omega\vp_\ve}<C
\end{equation*}
for some constant $C>0$ which depends only on the geometry of $(M,\omega)$ and the function $f$ by the Uniform Skoda Integrability Theorem due to Zeriahi (\cite{Zeriahi}).
\medskip

The final step is $C^{2,\alpha}(M)$-estimate. For $k\ge2$ and $\alpha\in(0,1)$, the standard Evans-Krylov method (\cite{Evans, Krylov}) and Schauder estimates (cf, see \cite{Aubin2, Gilbarg_Trudinger}) imply
\begin{equation}\label{E:Ck-estimate}
\norm{\vp_\ve}_{C^{k,\alpha}(X)}
\le
C,
\end{equation}
where $C$ is a positive constant which depends only on $k,\alpha$, the geometry of $(M,\omega)$ and the function $f$. This completes the proof.
\end{proof}

Proposition \ref{P:approximation1} implies that there exists a $\hat\vp\in C^\infty(M)$ such that $\vp_\ve\rightarrow\hat\vp$ as $\ve\rightarrow0$ by passing through a subsequence. However, $\vp_\ve$ converges without choosing a subsequence.
\begin{corollary}\label{C:convergence_vp}
The solution $\vp_\ve$ converges to $\vp$ which satisfies the following normalization condition
\begin{equation*}
\int_M\vp e^f\omega^n=0.
\end{equation*}
\end{corollary}

\begin{proof}
Let $\vp_0$ be the unique solution of \eqref{E:CMAE0} which satisfies that $\int_M\vp_0e^f\omega^n=1$.
It is enough to show that $\vp_\ve$ satisfies 
$\int_M\vp_\ve e^f\omega^n\rightarrow0$ as $\ve\rightarrow0$.
By \eqref{E:Ck-estimate}, we have
\begin{equation*}
e^{\ve\vp_\ve}
=
1+\ve\vp_\ve+o(\ve).
\end{equation*}
On the other hand, 
\begin{equation*}
1
=
\int_M\omega^n
=
\int_M\paren{\omega+dd^c\vp_\ve}^n
=
\int_M e^{\ve\vp_\ve} e^{f_\ve}\omega^n
=
\int_M
\paren{1+\ve\vp_\ve+o(\ve)}
e^{f_\ve}\omega^n,
\end{equation*}
so it follows that
\begin{equation*}
\ve\int_M\vp_\ve e^{f_\ve}\omega^n
=
o(\ve)\int_X e^{f_\ve}\omega^n.
\end{equation*}
This implies the conclusion.
\end{proof}

\subsection{Approximation on a family of complex Monge-Amp\'ere equations}
\label{SS:AFCMAE}
Let $p:X^{n+d}\rightarrow Y^d$ be a smooth family of compact K\"ahler manifolds and $\omega$ be a fixed K\"ahler form on $X$. Let $\xi$ be a differential form of degree $2n+r$ on $X$. 
The fiber integral is a differential form of degree $r$ on $Y$, which is defined as follows: Fix a point $y\in Y$ and let $(U,s=(s^1,\dots,s^d))$ be a coordinate centered at $y$ such that there exists a $C^\infty$ trivialization of the family:
\begin{equation*}
\Phi:X_0\times U\rightarrow p^{-1}(U).
\end{equation*}
In an admissible coordinate $(z,s)$, the pull-back $\Phi^*\xi$ is of the form
\begin{equation*}
\sum\xi_k(z,s) dV_z\wedge d\sigma^{k_1}\wedge\cdots\wedge d\sigma^{k_r},
\end{equation*}
where the $\sigma^{k_j}$ run through the real and imaginay parts of $s^j$ and $dV_z$ denotes the relative Euclidean volume form. Now the fiber integral is defined by
\begin{equation*}
\int_{X/Y}\xi
=
\int_{X_0\times Y/Y}
\Phi^*\xi
=
\sum\paren{\int_{X_s}\xi_k(z,s)dV_z}
d\sigma^{k_1}\wedge\cdots\wedge d\sigma^{k_r}.
\end{equation*}
Note that this definition is independent of the choice of coordinates and differentiable trivializations. The fiber integral coincides with the push-forward of the corresponding current. Hence, if $\xi$ is a differentiable form of type $(n+r,n+s)$, then the fiber integral is of type $(r,s)$. In particular, if $\xi$ be a differentiable form of type $(n,n)$ on $X$, then $\int_{X_s}\xi$ is a smooth function on $Y$. Moreover, we have the following properties (for the details, see \cite{Schumacher}.):
\begin{itemize}
\item [(\romannumeral1)] Fiber integration coincides with the push forward of a form, which is defined as follows:
For a form $\xi$ on $X$, $p_*\xi$ is defined by the form on $Y$ which satisfies
\begin{equation*}
\int_Y (p_*\xi)\wedge\zeta
=
\int_X \xi\wedge(p^*\zeta)
\end{equation*}
for any form $\zeta$ on $Y$.
\item [(\romannumeral2)]Fiber integration commutes with taking exterior derivatives:
	\begin{equation*}
	d\int_{X_s}\xi = \int_{X_s}d\xi
	\end{equation*}	 
\item [(\romannumeral3)] For a smooth form $\xi$ of type $(n,n)$, 
	\begin{equation*}
	\pd{}{s^i}\int_{X_s}\xi=\int_{X_s}L_V(\xi)
	\end{equation*}
	for any smooth lifting $V$ of $\partial/\partial s^i$ on $X$.
\end{itemize}
Note that the volume of a fiber does not change, namely, (\romannumeral2) implies that
\begin{equation*}
d\mathrm{Vol}_{\omega\vert_{X_s}}(X_s)
=
d\int_{X_s}\omega^n
=
\int_{X_s}d\omega^n
=0.
\end{equation*}
Hence we may assume that $\mathrm{Vol}_{\omega\vert_{X_y}}(X_y)=1$ for every $y\in Y$. The third property (\romannumeral3) will be used in Section \ref{S:app_geo_curv}. 
\medskip

From now on, we consider a smooth family $p:X\rightarrow\DD$ of compact K\"ahler manifolds over the unit disc $\DD$ in $\CC$. 
Let $\omega$ be a K\"ahler form on $X$.
Under an admissible coordinate $(z^1,\dots,z^n,s)$ in $X$, $\omega$ is written as follows:
\begin{equation}\label{E:local_omega}
\omega
=
\im\paren{g_{s\bar s}ds\wedge d\bar s
+g_{s\bar\beta}ds\wedge{dz}^{\bar\beta} 
+g_{\alpha\bar s}dz^\alpha\wedge d\bar s
+g_{\alpha\bar\beta}dz^\alpha\wedge{dz}^{\bar\beta}
}.
\end{equation}
For $0<\ve\le1$, let $\{f_\ve\}$ be a sequence of smooth functions on $X$. We consider the following fiberwise complex Monge-Amp\`ere equations:
\begin{equation} \label{E:ACMAE}
\begin{aligned}
\paren{\omega_y+dd^c\vp_y}^n
&=
e^{\ve\vp_y+f_\ve\vert_{X_y}}(\omega_y)^n,
\\
\omega_y+dd^c\vp_y
&>0
\end{aligned}
\end{equation}
on $X_y$ for $y\in\DD$. Theorem \ref{T:AY} implies that for each $y$, there exists a unique solution of \eqref{E:ACMAE}, call it $\vp_{y,\ve}\in C^\infty(X_y)$. It is remarkable to note that the function $\vp_\ve$ defined by 
\begin{equation*}
\vp_\ve(x)=\vp_{y,\ve}(x),
\end{equation*}
where $y=p(x)$, is a smooth function on $X$. This follows from the openness analysis of the continuity method for complex Monge-Amp\`ere equations and the implicit function theorem (\cite{Yau}). By Section \ref{SS:approximation1}, there exists a constant $C_y>0$ such that
\begin{equation}\label{E:single_estimate}
\norm{\vp_\ve}_{C^{k,\alpha}(X_y)}
\le
C_y
\end{equation}
where $C_y$ does not depend on $\ve$. Since we are now considering a local property on $y$, we may assume that $C=C_y$ does not depend on $y$.

In this section, we consider the $C^{k,\alpha}$-estimates for $V\vp_\ve$ and $\bar VV\vp_\ve$ on a fixed fiber $X_y$, where $V$ is any smooth lifting of $\partial/\partial s$ written as follows:
\begin{equation*}
V=\pd{}{s}+a\ind{s}{\gamma}{}\pd{}{z^\gamma}.
\end{equation*}
Before going further, we introduce the following proposition.

\begin{proposition}\label{P:Key_Prop}
Let $(X,\omega)$ be a compact K\"ahler manifold.
Let $\{\rho_\ve\}_{\ve\in I}$ be a family of K\"ahler metrics on $X$ which are uniformly equivalent to $\omega$, i.e., there exists a constant $C_1>0$ such that
\begin{equation*}
\frac{1}{C_1}\omega<\rho_\ve<C_1\omega
\;\;\;
\text{for all}\;\;\;
\ve\in I.
\end{equation*}
Let $u_\ve$ be a solution of the following PDE:
\begin{equation}\label{E:laplacian}
-\Delta_{\rho_\ve}u_\ve
+
\ve u_\ve
=
R_\ve,
\end{equation}
where $R_\ve$ is a smooth function on $X$ with 
$$
\norm{R_\ve}_{C^{k,\alpha}(X)}<C_2.
$$
Suppose that
\begin{equation*}
\abs{\int_X u_\ve\omega^n}<C_3.
\end{equation*}
Then there exists a uniform constant $C>0$ which depends only on $C_1$, $C_2$, $C_3$ and the geometry of $(X,\omega)$ such that
\begin{equation*}
\norm{u_\ve}_{C^{k,\alpha}(X)}<C.
\end{equation*}
\end{proposition}

\begin{proof}
In this proof, we shall use the Schauder estimate, Poincar\'e inequality and Sobolev inequality with respect to the K\"ahler metric $\rho_\ve$ (cf, see \cite{Gilbarg_Trudinger, Aubin2}). 
It is remakable to note that the constants in those inequalities do not depend on $\ve\in I$ since all $\rho_\ve$ are uniformly equivalent to $\omega$.
If we have the uniform estimate, i.e., $L^\infty$-estiamte of $u$, then Schauder estimate completes the proof. 
\medskip

The Poincar\'e inequality says that there exists a constant $C$ which depends only on $C_1$ and the geometry of $(M,\omega)$ such that
\begin{equation*}
\norm{u_\ve-\int_X u_\ve{\rho_\ve}^n}_{L^2_{\rho_\ve}(X)}
<
C\norm{Du_\ve}_{L^2_{\rho_\ve}(X)},
\end{equation*}
where $D$ is a total derivative.
It follows from the assumption that
\begin{equation*}
\norm{u_\ve}_{L^2_{\rho_\ve}(X)}
<C\norm{Du_\ve}_{L^2_{\rho_\ve}(X)}+C_1C_3.
\end{equation*}
On the other hand, multiplying $u_\ve$ to \eqref{E:laplacian} and integrating it with respect to $(\rho_\ve)^n$, we have
\begin{equation*}
\norm{Du_\ve}_{L^2_{\rho_\ve}(X)}^2
+
\ve\norm{u_\ve}_{L^2_{\rho_\ve}(X)}^2
=
\int_X R_\ve u_\ve{\rho_\ve}^n.
\end{equation*}
The H\"older inequality says that
\begin{equation}
\norm{Du_\ve}_{L^2_{\rho_\ve}(X)}^2
\le
\norm{R_\ve}_{L^2_{\rho_\ve}(X)}
\norm{u_\ve}_{L^2_{\rho_\ve}(X)}
\end{equation}
Combining the two equations, there exists a uniform constant $C$ which depends only on $C_1, C_2, C_3$ and the geometry of $(X,\omega)$ such that
\begin{equation*}
\norm{u_\ve}_{L^2_{\rho_\ve}(X)}<C.
\end{equation*}

Now we follow the Moser iteration. 
Multiplying \eqref{E:pde_vpve'} by $\abs{u_\ve}^{2p-1}\cdot u_\ve/\abs{u_\ve}$ and integrating it, we have
\begin{equation*}
\frac{2p-1}{p^2}\int_{X_y}\abs{D\abs{u_\ve}^p}^2\omega^n
+
\ve\int_{X}\abs{u_\ve}^{2p}\omega^n
=
\int_{X} R_\ve u_\ve\omega^n.
\end{equation*}
The Sobolev inequality says that
\begin{equation*}
\norm{\abs{u_\ve}^p}^2_{L^{2n/(n-1)}_{\rho_\ve}(X)}
\le
C
\paren{
	\norm{\abs{u_\ve}^p}_{L^2_{\rho_\ve}(X)}
	+
	\norm{D\abs{u_\ve}^p}_{L^2_{\rho_\ve}(X)}
}
\end{equation*}
for $p\ge1$ (\cite{Aubin2}). Combining two equations, we have
\begin{equation*}
\norm{u_\ve}_{L^{2p\cdot\frac{n}{n-1}}_{\rho_\ve}(X)}
\le
(Cp)^{1/p}\norm{u_\ve}_{L^{2p}_{\rho_\ve}(X)}
\end{equation*}
for $p\ge1$.
The uniform estimate is obtained by the Moser iteration method (cf, see \cite{Gilbarg_Trudinger}). Indeed, set
\begin{equation*}
p_1=1,
\;\;\;
p_k=\paren{\frac{n}{n-1}}^k.
\end{equation*}
Then it follows that
\begin{equation*}
\norm{u}_{L^\infty(X)}
=
\lim_{k\rightarrow\infty}\norm{u_\ve}_{L^{2p_k}_{\rho_\ve}(X)}
\le \prod_{k=1}^n(Cp_k)^{1/p_k}\norm{u_\ve}_{L^2_{\rho_\ve}(X)}.
\end{equation*}
This completes the proof.
\end{proof}

\begin{proposition}\label{P:approximation2}
Suppose that there exist constants $C_1>0$ and $C_2>0$ such that
\begin{equation*}
\abs{\int_{X_y}(V\vp_\ve)(\omega_y)^n}<C_1
\end{equation*}
and
\begin{equation*}
\norm{Vf_\ve}_{C^{k,\alpha}(X_y)}<C_2.
\end{equation*}
Then there exits a constant $C$ which depends only on the constants $C_1$, $C_2$, the lift $V$ and the geometry of $(X_y,\omega_y)$ such that
\begin{equation*}
\norm{V\vp_\ve}_{C^{k,\alpha}(X_y)}<C
\end{equation*}
for $0<\ve\le1$. In particular, $\{V\vp_\ve\}_{0<\ve\le1}$ is a relatively compact subset in $C^{k,\alpha}(X_y)$ for any $k\in\NN$ and $\alpha\in(0,1)$.
\end{proposition}

\begin{proof}
We denote by $\rho_\ve=\omega+dd^c\vp_\ve$. Note that Proposition \ref{P:approximation1} implies that there exists a uniform constant $C>0$ such that
\begin{equation}\label{E:equivalence}
\frac{1}{C}\omega_y
<
\rho_\ve\vert_{X_y}
<
C\omega_y,
\end{equation}
for $0<\ve\le1$. Under an admissible coordinate $(z^1,\dots,z^n,s)$, the first equation of \eqref{E:ACMAE} is written as follows:
\begin{equation} \label{E:CMAE_coord}
\det(g_{\alpha\bar\beta}+(\varphi_\ve)_{\alpha\bar\beta})
=
e^{\varepsilon\varphi_{\varepsilon}+f_\ve}
\det(g_{\alpha\bar\beta})
\end{equation}
on each $X_y$. Taking logarithm of \eqref{E:CMAE_coord} and differentiating it with respect to $V$, we have
\begin{equation*}
(\rho_\ve)^{\alpha\bar\beta}
V\paren{g_{\alpha\bar\beta}+(\vp_\ve)_{\alpha\bar\beta}}
=
\ve V\vp_\ve
+
V f_\ve
+
g^{\alpha\bar\beta}
V\paren{g_{\alpha\bar\beta}}.
\end{equation*}
For a smooth function $\xi$, we denote by
\begin{align*}
[V,\xi]_{\alpha\bar\beta}
&=
V(\xi_{\alpha\bar\beta})-(V\xi)_{\alpha\bar\beta} \\
&=
-a\ind{s}{\gamma}{\alpha\bar\beta}\xi_\gamma
-a\ind{s}{\gamma}{\alpha}\xi_{\gamma\bar\beta}
-a\ind{s}{\gamma}{\bar\beta}\xi_{\alpha\gamma}.
\end{align*}
It is remarkable to note that $[V,\xi]_{\alpha\bar\beta}$ does not include $s$-derivative of $\xi$.
Then it follows that
\begin{equation}\label{E:pde_vpve}
\begin{aligned}
-\Delta_{\rho_\ve\vert_{X_y}}\paren{V\vp_\ve}
+\ve\paren{V\vp_\ve}
=&
-V f_\ve
-
g^{\alpha\bar\beta} 
V\paren{g_{\alpha\bar\beta}}\\
&+
(\rho_\ve)^{\alpha\bar\beta}
\paren{
V\paren{g_{\alpha\bar\beta}}
+
[V,\vp_\ve]_{\alpha\bar\beta}
}
\end{aligned}
\end{equation}
on each fiber $X_y$, where $\Delta_{\rho_\ve\vert_{X_y}}$ is the Laplace-Beltrami operator on $X_y$ with respect to $\rho_\ve\vert_{X_y}$. Here $(V\vp_\ve)$ and $(Vf_\ve)$ mean that
\begin{equation*}
V\vp_\ve
=
(V\vp_\ve)\vert_{X_y}
\;\;\;\text{and}\;\;\;
V f_\ve
=
(Vf_\ve)\vert_{X_y}.
\end{equation*}
From now on, when we think about a family of PDEs, we omit the subsrcript $X_y$ in the Laplace-Beltrami opertor, i.e., we write as follows:
\begin{equation*}
\Delta_{\rho_\ve}=\Delta_{\rho_\ve\vert_{X_y}}.
\end{equation*}
Equation \eqref{E:pde_vpve} says that the right hand side of \eqref{E:pde_vpve} is a globally defined function on $X_y$, call it $R_\ve$. Then we have
\begin{equation}\label{E:pde_vpve'}
-\Delta_{\rho_\ve}(V\vp_\ve)
+
\ve(V\vp_\ve)
=
R_\ve.
\end{equation}
This is a second order elliptic partial differential equation with the hypotheses in Proposition \ref{P:Key_Prop}. This completes the proof.
\end{proof}

\begin{proposition}\label{P:approximation3}
Under the assumption in Proposition  \ref{P:approximation2}, suppose that there exists a constant $C_3>0$ and $C_4$ such that
\begin{equation*}
\abs{\int_{X_y}\paren{\bar VV\vp_\ve}(\omega_y)^n}<C_3
\end{equation*}
and
\begin{equation*}
\norm{\bar VVf_\ve}_{C^{k,\alpha}(X_y)}<C_4.
\end{equation*}
Then there exits a constant $C$ which depends only on constants $C_1, C_2, C_3, C_4$, the lift $V$ and the geometry of $(X_y,\omega_y)$ such that
\begin{equation*}
\norm{\bar VV\vp_\ve}_{C^{k,\alpha}(X_y)}<C
\end{equation*}
for $0<\ve\le1$. In particular, $\{\bar VV\vp_\ve\}_{0<\ve\le1}$ is a relatively compact subset in $C^{k,\alpha}(X_y)$ for any $k\in\NN$ and $\alpha\in(0,1)$.
\end{proposition}

\begin{proof}
Differentiating \eqref{E:pde_vpve'} with respect to $\bar V$, we have
\begin{equation*}
-\Delta_{\rho_\ve}\paren{\bar VV\vp_\ve}
+
\ve\paren{\bar VV\vp_\ve}
=
\bar V\paren{(\rho_\ve)^{\bar\beta\alpha}}\cdot\paren{V\vp_\ve}_{\alpha\bar\beta}
+
(\rho_\ve)^{\bar\beta\alpha}[\bar V,V\vp_\ve]_{\alpha\bar\beta}
+
\bar V(R_\ve).
\end{equation*}
Since $\norm{\vp_\ve}_{C^{k,\alpha}(X_y)}$ and $\norm{V\vp_\ve}_{C^{k,\alpha}(X_y)}$ are bounded, the same argument in the proof of Propostion \ref{P:approximation2} says this PDE satisfies the hypotheses in Proposition \ref{P:Key_Prop}. This completes the proof.
\end{proof}

\section{Fiberwise Ricci-flat metrics on Calabi-Yau fibrations}\label{S:fiberwiseRFf}

In this section, we discuss the properties of the fiberwise Ricci-flat metric $\rho$. We first discuss a partial differential equation which the geodesic curvature $c(\rho)$ satisfies and several applications of this PDE.
\medskip

Let $p:X\rightarrow Y$ be a smooth family of Calabi-Yau manfiolds  and $\omega$ be a K\"ahler form  on $X$. We write $\omega$ like as \eqref{E:local_omega}. Since every fiber $X_y$ is a Calabi-Yau manifold, the first Chern class $c_1(X_y)$ vanishes for each fiber $X_y$. Since $c_1(X_y)$ is represented by the Ricci form of $\omega_y$, we know that 
\begin{equation*}
\left[
	-dd^c\log\det(g_{\alpha\bar\beta}(\cdot,y))
\right]
=0.
\end{equation*}
By the $dd^c$-lemma, there exists a unique function $\eta_y\in C^\infty(X_y)$ such that
\begin{itemize}
\item $\displaystyle dd^c\eta_y=dd^c\log\det(g_{\alpha\bar\beta})$ and
\item $\int_{X_y} e^{\eta_y}(\omega_y)^n=\int_{X_y}(\omega_y)^n$.
\end{itemize}
For each $y\in Y$, there exists a unique solution $\vp_y\in C^\infty(X_y)$ of the following complex Monge-Amp\`ere equation on each fiber $X_y$:
\begin{equation}
\begin{aligned}
\paren{\omega_y+dd^c\varphi_y}^n
&=
e^{\eta_y}(\omega_y)^n, \\
\omega_y+dd^c\varphi_y
&>
0,
\end{aligned}
\end{equation}
which is normalized by $\int_{X_y}\vp_ye^{\eta_y}(\omega_y)^n=0$.
Then it is easy to see that $\omega_y+dd^c\varphi_y$ is the Ricci-flat K\"ahler metric on $X_y$. As we already mentioned, we can consider $\vp$ as a smooth function on $X$ by letting $\vp(x)=\vp_y(x)$ where $y=p(x)$. Define a real $(1,1)$-form $\rho$ on $X$ by
\begin{equation*}
\rho=\omega+dd^c\vp.
\end{equation*}
Since $e^{\eta_y}(\omega_y)^n=(\omega_{KE,y})^n$, this is the fiberwise Ricci-flat metric in Theorem \ref{T:main_theorem}. 
\medskip

\medskip

The following theorem is proved in \cite{Braun:Choi:Schumacher}. 
(Cf, see Theorem \ref{T:PDE}.)

\begin{theorem}\label{T:PDE0}
Let $V\in T_yY$. Then the following PDE holds on $X_y$:
\begin{equation}\label{E:PDE0}
-\Delta_\rho c(\rho)(V)
=
\abs{\bar\partial V_\rho}_\rho^2
-
\Theta_\rho(V_\rho,\bar V_\rho).
\end{equation}
\end{theorem}

In \cite{Braun:Choi:Schumacher}, it is proved that the curvature form $\Theta_\rho$ of a family of compact K\"ahler manifolds with vanishing first Chern class  satisfies that on a admissible coordinate $(z^1,\dots,z^n,s^1,\dots,s^d)$ we have
\begin{equation*}
\Theta_{\alpha\bar\beta}=0,
\;\;
\Theta_{s\bar\beta}=0,
\;\;
\Theta_{\alpha\bar j}=0,
\;\;\text{and}\;\;
\Theta_{i\bar j}(z,s)=\Theta_{i\bar j}(s).
\end{equation*}
In particular, $\Theta_\rho(V_\rho,\bar V_\rho)$ is a constant on each fiber $X_y$.
However, in case of a family of compact K\"ahler manifold with trivial line bundle, the curvature $\Theta_\rho$ coincides with the curvature of $E=p_*(K_{X/y})$ as follows:
\medskip

Since every fiber $X_y$ is Calabi-Yau, $K_{X_y}$ is a trivial line bundle for every $y\in Y$.
Hence the direct image bundle $E=p_*(K_{X/Y})$ is a line bundle over $Y$.
Take an admissible coordinate system $(z^1,\dots,z^n,s^1,\dots,s^d)$ in $X$.
Let $u$ be a local holomorphic section of $E$ over an open set $U\subset Y$.
(Shrinking $U$ if necessary, $s=(s^1,\dots,s^d)$ can be considered as a local coordinate in $U$.) Since $E$ is a line bundle, the curvature of $(E,\norm\cdot)$ is given by
\begin{equation*}
\Theta(E)=-dd^c\log\norm{u}_s.
\end{equation*}
We say that $\mathbf{u}$ is a representative of $u$ if $\mathbf{u}$ is an $(n,0)$-form on $p^{-1}(U)$, such that $\mathbf{u}$ restricts to $u_s$ on fibers $X_s$, i.e.,
\begin{equation*}
\iota_s^*(\mathbf{u}) = u_s
\end{equation*}
where $\iota_s$ is the natural inclusion map from $X_s$ to $X$ (for more details, see \cite{Berndtsson1, Berndtsson2}).
The representative is not uniquely determined, but any two representatives are differ from $ds\wedge v$ for some $(n-1,0)$-form $v$.
Hence if we denote by $u\wedge\overline u\wedge dV_s:=\mathbf{u}\wedge\overline{\mathbf{u}}\wedge dV_s$, where $dV_s=c_d ds\wedge d\bar s$, then it does not depend on the choice of the representative. Moreover, it also follows that
\begin{equation*}
\displaystyle{\norm{u}^2_s=c_n\int_{X_s}u\wedge\overline u=c_n\int_{X_s}\mathbf{u}\wedge\overline{\mathbf{u}}}
\end{equation*}
for any representative $\mathbf{u}$ of $u$. In terms of $u$, the function $\eta$ is written explicitly:

\begin{proposition} \label{P:key_p}
On $p^{-1}(U)$, $\eta$ is written as follows:
\begin{equation}\label{E:eta}
\eta(z,s)
=
-\log\frac{\omega^n\wedge dV_s}
	{c_n u\wedge \overline{u}\wedge dV_s}
-\log\frac{\norm{u}^2_s}{\vol(X_s)}.
\end{equation}
In particular, we have the following:
\begin{equation*}
\Theta_{h^\rho_{X/Y}}(K_{X/Y})
=
p^*\Theta(E).
\end{equation*}
\end{proposition}

\begin{proof}
Let $\mathbf{u}$ be a representative of $u$. Denote the right hand side of \eqref{E:eta} by $\tilde\eta$.
It is enough to show the following:
\begin{itemize}
\item [1.]$\int_{X_s}e^{\tilde\eta}(\omega_s)^n=\int_{X_s}(\omega_s)^n$.
\item [2.]$dd^c\tilde\eta\vert_{X_s}=-dd^c\log\det(g_{\alpha\bar\beta})\vert_{X_s}.$
\end{itemize}
First we compute
\begin{align*}
\int_{X_s}e^{\tilde\eta}(\omega_s)^n
&=
\int_{X_s}\left[\exp\paren{-\log\frac{\omega^n\wedge dV_s}
	{c_n \mathbf{u}\wedge \overline{\mathbf{u}}\wedge dV_s}
-\log\frac{\norm{u}^2_s}{\vol(X_s)}}\right](\omega_s)^n.
\end{align*}
If we write $dz=dz^1\wedge\dots\wedge dz^n$, then 
\begin{equation*}
(\omega_s)^n=\det(g_{\alpha\bar\beta})c_n dz\wedge d\bar z
\;\;\;\text{and}\;\;\;
\mathbf{u}\vert_{X_s}=\hat u(z,s)dz
\end{equation*}
for some local holomorphic function $\hat u(z,s)$. 
It follows that
\begin{align*}
\int_{X_s}e^{\tilde\eta}(\omega_s)^n
&=
\int_{X_s}\exp
\paren{-\log\frac{\det(g_{\alpha\bar\beta})}{c_n\abs{\hat u(z,s)}^2}
-
\log\frac{\norm{u}^2_s}{\vol(X_s)}
}
(\omega_s)^n \\
&=
\frac{\vol(X_s)}{\norm{u}_s^2}
\cdot c_n\int_{X_s}
\frac{\abs{\hat u(z,s)}^2}{\det(g_{\alpha\bar\beta})}
{\det(g_{\alpha\bar\beta})}dz\wedge d\bar z \\
&=
\frac{\vol(X_s)}{\norm{u}_s^2}
\cdot c_n\int_{X_s}
\hat u(z,s)dz\wedge \overline{\hat u(z,s)dz} \\
&=
\frac{\vol(X_s)}{\norm{u}_s^2}
\cdot c_n\int_{X_s}
\mathbf{u}\wedge\overline{\mathbf{u}}
=
\int_{X_s}(\omega_s)^n.
\end{align*}
Moreover, we have
\begin{align*}
dd^c\tilde\eta\vert_{X_s}
&=
-dd^c
\paren{\log\det(g_{\alpha\bar\beta})
	+\log\abs{\hat u(z,s)}^2
}\Big\vert_{X_s} 
\\
&=
-dd^c\log\det(g_{\alpha\bar\beta})\vert_{X_s}.
\end{align*}
This yields the first assertion. For the second assertion,
\begin{align*}
\Theta_\rho
&=
dd^c\log \rho^n\wedge dV_s
=
dd^c\log e^\eta\omega^n\wedge dV_s\\
&=
dd^c\log c_n u\wedge \overline{u}\wedge dV_s
-
dd^c\log\norm{u}^2_s
=
p^*\Theta(E).
\end{align*}
This completes the proof.
\end{proof}

\begin{remark}\label{R:sufficient_condition}
To show that $p_*\rho^{n+1}$ is positive on $Y$, it is enough to consider a Calabi-Yau fibration over the unit disc by the following: 
\begin{itemize}
\item [1.] Let $\sigma_1$ and $\sigma_2$ be real $(1,1)$-forms on $Y$. Suppose that
\begin{equation*}
p_*(\sigma_1\vert_{\gamma(\DD)})^{n+1}
\ge
p_*(\sigma_2\vert_{\gamma(\DD)})^{n+1}
\end{equation*}
for each holomorphic disc $\gamma^{n+1}:\DD\rightarrow Y$.
Then we have
$p_*(\sigma_1)^{n+1}\ge p_*(\sigma_2)^{n+1}$ on $X$.
\item [2.] Every computation concerning the positivity of $p_*\rho^{n+1}$ is local in $s$-variable, which is a local coordinate in $Y$.
\end{itemize}
Therefore we only consider a famliy of Calabi-Yau manifolds over the unit disc in $\CC$ as long as we are interested in positivity properties of $p_*\rho^{n+1}$. In this case, \eqref{E:PDE0} turns out to be
\begin{equation}\label{E:PDE0'}
-\Delta_\rho c(\rho)
=
\abs{\bar\partial v_\rho}_\rho^2-\Theta_{s\bar s},
\end{equation}
where $v=\partial/\partial s$ and $\Theta_{s\bar s}:=\Theta_\rho(v_\rho,\bar v_\rho)$. As we mentioned in Section \ref{SS:horizontal_lift}, the positivity of $p_*\rho^{n+1}$ is equivalent to $\int_{X_y}c(\rho)\rho^n>0$. 
\end{remark}

\begin{remark}
In case of a family of canonically polarized compact complex manifolds $p:X\rightarrow\DD$, Schumacher have proved that the geodesic curvature $c(\tilde\rho)$ of the form $\tilde\rho$, which is induced by the fiberwise K\"ahler-Einstein metrics of Ricci curvature $-1$, satisfies the following PDE:
\begin{equation}\label{E:Schumacher}
-\Delta_\rho c(\tilde\rho)+c(\tilde\rho)
=
\abs{\bar\partial v_{\tilde\rho}}_{\tilde\rho}^2
\end{equation}
for each fiber $X_y$ (\cite{Schumacher}). This PDE gives a lower bound of $c(\tilde\rho)$ directly by the maximum principle. (Moreover, a lower bound is also obtained using heat kernel estimates.) Hence the fiberwise K\"ahler-Einstein form $\tilde\rho$ is a semi-positive metric on $X$. However \eqref{E:PDE0'} does not gives a lower bound by the maximum principle. 
\end{remark}

In the last of this section, we discuss some applications of Theorem \ref{T:PDE0}.

\begin{proposition}\label{P:harmonic_representative}
$\bar\partial V_\rho\cdot u_y$ is the harmonic representative of the cohomology class $K_y(V)\cdot u_y$ with respect to $\rho\vert_{X_y}$.
\end{proposition}

\begin{proof}
Since $E$ is a line bundle, Griffiths' theorem implies that
\begin{equation*}
\Theta_{V\bar V}(E)
=
\frac{\norm{K_y(V)\cdot u_y}^2}{\norm{u_y}^2}.
\end{equation*}
Note that
\begin{equation*}
\bar\partial V_\rho
\in
K_y(V).
\end{equation*}
It follows that
\begin{equation*}
\frac{\norm{K_y(V)\cdot u_y}^2}{\norm{u_y}^2}
\le
\frac{\norm{\bar\partial V_\rho\cdot u_y}^2}{\norm{u_y}^2}.
\end{equation*}
The following lemma is well-
known (cf, see \cite{Popovici}).
\begin{lemma}\label{L:nonvanishingform}
Let $(X,\omega)$ be a Calabi-Yau manifold. Let $u$ be a non-vanishing holomorphic $n$-form on $X$ such that 
\begin{equation*}
\norm{u}^2_\omega
:=\int_X \abs{u}^2_\omega\;dV_\omega
=\int_X dV_\omega.
\end{equation*}
Denote by $A^{(p,q)}(E)$ the space of smooth $(p,q)$-forms with values in $E$. Define a map 
$$
T_u:A^{(0,1)}(T'X)\rightarrow A^{(n-1,1)}(X)
$$
by $T_u(V)=V\cdot u$.
Then $T_u$ is an isometry with respect to the pointwise scalar product induced by $\omega$.
\end{lemma}
Hence Lemma  implies that
\begin{equation*}
\norm{\bar\partial V_\rho}_\rho^2
=
\Theta_{V\bar V}(E)
=
\frac{\norm{K_y(V)\cdot u_y}^2}{\norm{u_y}^2}
\le
\frac{\norm{\bar\partial V_\rho\cdot u_y}^2}{\norm{u_y}^2}
=
\norm{\bar\partial V_\rho}_\rho^2.
\end{equation*}
It follows that $\bar\partial V_\rho\cdot u_y$ is the harmonic representative with respect to $\rho\vert_{X_y}$ of $K_y(V)\cdot u_y$. This completes the proof.
\end{proof}

\begin{proposition}
Let $p:X\rightarrow Y$ be a Calabi-Yau fibration. If the curvature $\Theta_{h^\rho_{X/Y}}(K_{X/Y})$ vanishes along a complex curve, then the fibration is trivial along the complex curve.
\end{proposition}

\begin{proof}
Denote by $\gamma$ the complex curve in $Y$. Then $p\vert_\gamma:X_\gamma\rightarrow\gamma$ is a Calabi-Yau fibration over a $1$-dimensional base. If we take $s$ be a holomorphic coordinate of $\gamma$, then we have Equation \eqref{E:PDE0'} on each fiber $X_y$ for $y\in\gamma$. By the Hypothesis, $\Theta_\rho(v_\rho,\ov{v_\rho})$ vanishes on $\gamma$. Integrating \eqref{T:PDE0}, we know that $v_\rho$ is a holomorphic vector field on $X_\gamma$. The flow of $v_\rho$ makes $X_\gamma$ a trivial fibration.
\end{proof}

\section{Proof of Theorem \ref{T:main_theorem}}
\label{S:proof}

In this section we shall prove the main theorem. As we mentioned in Remark \ref{R:sufficient_condition}, it is enough to show that $\int_{X/\DD}c(\rho)\rho^n>0$ for a family of Calabi-Yau manifolds over the unit disc in $\CC$.

Let $p:X\rightarrow\DD$ be a smooth family of Calabi-Yau manifolds. For each $\ve>0$, we consider the following fiberwise complex Monge-Amp\`ere equation on each fiber $X_y$:
\begin{equation}\label{E:PDE1'}
\begin{aligned}
\paren{\omega_y+dd^c\vp_y}^n
&=
e^{\varepsilon\vp_y}e^{\eta_y}(\omega_y)^n
\;\;\text{and}
\\
\omega_y+dd^c\vp_y
&>0,
\end{aligned}
\end{equation}
where $\eta$ is defined in Section \ref{S:fiberwiseRFf}.
Theorem \ref{T:AY} implies that there exists a unique solution $\vp_{y,\ve}\in C^\infty(X_y)$ of \eqref{E:PDE1'}.
As we mentioned, we can consider $\vp_\ve$ as a smooth function on $X$ by letting $\vp_\ve(x):=\vp_{y,\ve}(x)$, where $y=p(x)$.
We consider next the $(1,1)$-form
\begin{equation} \label{E:rho}
\rho_\varepsilon
	:=\omega+dd^c\varphi_\varepsilon
\end{equation}
on the manifold $X$. Since $\rho_\varepsilon$ is positive definite when restricted to $X_y$, it induces a hermitian metric $h^{\rho_\varepsilon}_{X/Y}$ on the bundle $K_{X/Y}\vert_{X_0}$.
The curvature $\Theta_{h^{\rho_\varepsilon}_{X/Y}}(K_{X/Y})$ is computed as follows:
\begin{align*}
%\Theta_{h^{\rho_\varepsilon}_{X/Y}}(K_{X/Y})
\Theta_{h^{\rho_\varepsilon}_{X/Y}}(K_{X/Y})
& =
dd^c\log
\paren{
	(\rho_\ve)^n\wedge\im ds\wedge d\bar s
}
 =
dd^c\log
\paren{
	e^{\varepsilon\varphi_\epsilon}
	e^\eta\omega^n\wedge\im ds\wedge d\bar s} \\
& =
\ve dd^c\varphi_\ve
+
\Theta_{h^\rho_{X/Y}}(K_{X/Y}) 
\end{align*}
From \eqref{E:rho}, we have $dd^c\varphi_\varepsilon=\rho_\varepsilon-\omega$,
it follows that
\begin{equation}\label{E:Ricci}
\Theta_{\rho_\varepsilon}
=
\varepsilon{\rho_\varepsilon}
-
\varepsilon\omega
+
\Theta_\rho
\end{equation}
in another expression,
\begin{equation*}
\ve{\rho_\varepsilon}
=
\ve\omega
+
\Theta_{\rho_\varepsilon}
-
\Theta_\rho.
\end{equation*}
Our next claim is the geodesic curvature $c(\rho_\ve)$ satisfies a certain elliptic partial differential equation of second order on each fiber $X_y$.
\medskip

Under an admissible coordinate $(z^1,\dots,z^n,s)\in X$, $\rho_\ve$ is written as follows:
\begin{equation*}
\rho_\ve
=
\im\paren{(h_\ve)_{s\bar s}ds\wedge d\bar s
+(h_\ve)_{s\bar\beta}ds\wedge{dz}^{\bar\beta}
+(h_\ve)_{\alpha\bar s}dz^\alpha\wedge d\bar s
+(h_\ve)_{\alpha\bar\beta}dz^\alpha\wedge{dz}^{\bar\beta}
}.
\end{equation*}
For each $y\in\DD$, $(h_\ve)_{\alpha\bar\beta}(\cdot,y)$ gives a K\"{a}hler metric on $X_y$. (If there is no confusion, we simply write $(h_\ve)_{\alpha\bar\beta}$.) Thus we can define contraction and covariant derivative on each $X_y$ with respect to $(h_\ve)_{\alpha\bar\beta}$. We use raising and lowering of indices as well as the semi-colon for the contractions and the covariant derivatives with respect to the K\"{a}hler metric $(h_\ve)_{\alpha\bar\beta}$, respectively, on the fiber $X_y$.  We denote by $\Delta_{\rho_\ve}=\Delta_{\rho_\ve\vert_{X_y}}$ the Laplace-Beltrami operator with negative eigenvalues on the fiber $X_y$ with respect to $\rho_\ve\vert_{X_y}$.

By raising of indices, we can write the horizontal lift $v_{\rho_\varepsilon}$ of $v=\partial/\partial s$ with respect to $\rho_\varepsilon$ by
\begin{equation*}
v_{\rho_\varepsilon}
=
\pd{}{s}
-(h_\ve)_{s\bar\beta}(h_\ve)^{\bar\beta\alpha}\pd{}{z^\alpha}
=
\pd{}{s}
-(h_\ve)\ind{s}{\alpha}{}\pd{}{z^\alpha}.
\end{equation*}
Then Remark \ref{R:horizontal_lift} says that the geodesic curvature $c(\rho_\ve):X\rightarrow\RR$ is given by
\begin{align*}
c(\rho_\ve)(z,s)
=
\inner{v_{\rho_\varepsilon},v_{\rho_\varepsilon}}_{\rho_\varepsilon}
=
(h_\ve)_{s\bar{s}}
-
(h_\ve)_{s\bar\beta}(h_\ve)^{\bar\beta\alpha}(h_\ve)_{\alpha\bar{s}}.
\end{align*}
It is remarkable to note that $\bar\partial{v_{\rho_\varepsilon}}$ is a representative of the Kodaira-Spencer class which is is a $T'X_y$-valued $(0,1)$-form which is defined by
\begin{align*}
\bar\partial{v_{\rho_\varepsilon}}
=
\bar\partial\paren{\pd{}{s}
-(h_\ve)_{s}^{\phantom{i}\alpha}\pd{}{z^\alpha}} 
=
-\pd{(h_\ve)\ind{s}{\alpha}{}}{z^{\bar\beta}}dz^{\bar\beta}\otimes\pd{}{z^\alpha}
=
-(h_\ve)\ind{s}{\alpha}{;\bar\beta}dz^{\bar\beta}\otimes\pd{}{z^\alpha}.
\end{align*}

The following theorem is inspired by Schumacher in \cite{Schumacher}. P\v aun generalized the computation to the twisted K\"ahler-Einstein metric case (\cite{Paun2}). (See also \cite{Choi1}.)

\begin{theorem}\label{T:PDE}
The following partial differential equation holds on each fiber $X_y$:
\begin{equation*}
-\Delta_{\rho_\ve} c(\rho_\varepsilon)
+
\varepsilon c(\rho_\varepsilon)
=
\varepsilon\omega(v_{\rho_\ve},\overline{v_{\rho_\ve}})
+
\abs{\bar\partial v_{\rho_\ve}}_{\rho_\ve}^2
-
\Theta_{s\bar s},
\end{equation*}
where $\abs{\bar\partial v_{\rho_\ve}}_{\rho_\ve}$ is the pointwise norm of $\bar\partial v_{\rho_\ve}$ with respect to the K\"ahler metric $\rho_\ve\vert_{X_y}$.
\end{theorem}

\begin{proof}
We fix a fiber $X_y$ and $\ve>0$. During this proof, if there is no confusion, we omit the subscript $\ve$ in the components in $\rho_\ve$ for simplicity, namely, we write as follows:
\begin{equation*}
h_{s\bar s}=(h_\ve)_{s\bar s},
\quad
h_{s\bar\beta}=(h_\ve)_{s\bar\beta}
\quad\text{and}\quad
h_{\alpha\bar\beta}=(h_\ve)_{\alpha\bar\beta}.
\end{equation*}

We have to compute the following:
\begin{equation*}
\Delta_{\rho_\ve} c(\rho_\ve)
	=h^{\bar\delta\gamma}(c(\rho_\ve))_{;\gamma\bar\delta}
	=h^{\bar\delta\gamma}\paren{
	h_{s\bar{s}}-h_{s\bar\beta}h^{\bar\beta\alpha}h_{\alpha\bar{s}}
	}_{;\gamma\bar\delta}.
\end{equation*}
First we consider the term $h^{\bar\delta\gamma}h_{s\bar{s};\gamma\bar\delta}$.
Since $\rho_\ve$ is locally $\ddbar$-exact, we have
\begin{align*}
h_{s\bar{s};\gamma\bar\delta}
	& =\pd{^2h_{s\bar{s}}}{z^\gamma\partial{z}^{\bar\delta}}
	= \pd{^2}{s\partial\bar{s}}h_{\gamma\bar\delta}.
\end{align*}
Then it follows that
\begin{align*}
h^{\bar\delta\gamma}h_{s\bar{s};\gamma\bar\delta}
	& = h^{\bar\delta\gamma}\pd{^2}{s\partial\bar{s}}h_{\gamma\bar\delta} 
	= \pd{}{s}\paren{h^{\bar\delta\gamma}\pd{}{\bar{s}}h_{\gamma\bar\delta}}
	- \pd{h^{\bar\delta\gamma}}{s}\pd{h_{\gamma\bar\delta}}{\bar{s}}
	\\
	& =\pd{^2}{s\partial\bar{s}}\log{\det(h_{\alpha\bar\beta})}
	+h^{\bar\delta\alpha}\pd{h_{\alpha\bar\beta}}{s}
	h^{\bar\beta\gamma}\pd{h_{\gamma\bar\delta}}{\bar{s}}
\end{align*}
By \eqref{E:Ricci}, we have
\begin{equation*}
\pd{^2}{s\partial\bar{s}}\log{\det(h_{\alpha\bar\beta})}
=
\ve\rho_\ve\paren{\pd{}{s},\pd{}{\bar s}}
-
\ve\omega\paren{\pd{}{s},\pd{}{\bar s}}
+
\Theta_{s\bar s}.
\end{equation*}
Hence it follows that
\begin{equation}\label{E:first_term}
h^{\bar\delta\gamma}h_{s\bar{s};\gamma\bar\delta}
=
\ve
\paren{
	h_{s\bar s}
	-
	g_{s\bar s}
}
+
\Theta_{s\bar s}
+
h_{s\bar\beta;\alpha}
h_{\bar{s}\gamma;\bar\delta}
h^{\bar\beta\gamma}h^{\bar\delta\alpha}.
\end{equation}

Next we consider the term
$h^{\bar\delta\gamma}\paren{h_{s\bar\beta}h^{\bar\beta\alpha}h_{\alpha\bar{s}}}_{;\gamma\bar\delta}$, which can be written by
$$
h^{\bar\delta\gamma}\paren{h\ind{s}{\alpha}{} h_{\alpha\bar{s}}}_{;\gamma\bar\delta}.
$$
Define a tensor $\{A\ind{s}{\alpha}{\bar\beta}\}$ by
$$
A\ind{s}{\alpha}{\bar\beta}=-h\ind{s}{\alpha}{;\bar\beta}.
$$
Then it follows that
\begin{equation*}
\bar\partial{v_\rho}=A\ind{s}{\alpha}{\bar\beta}\pd{}{z^\alpha}\otimes{dz}^{\bar\beta}.
\end{equation*}
Hence we have
\begin{align*}
h^{\bar\delta\gamma}\paren{h\ind{s}{\sigma}{}h_{\bar{s}\delta}}_{;\gamma\bar\delta}
	& = h^{\bar\delta\gamma}\paren{
	h\ind{s}{\sigma}{;\gamma\bar\delta}h_{\bar{s}\sigma}
	+A\ind{s}{\sigma}{\bar\delta}A_{\bar{s}\sigma\gamma}
	+h\ind{s}{\sigma}{;\bar\delta}h_{\bar{s}\sigma;\bar\delta}
	+h\ind{s}{\sigma}{}A_{\bar{s}\sigma\gamma;\bar\delta}
	}
	\\
	& := I_1+I_2+I_3+I_4.
\end{align*}
First of all, it is obvious that
\begin{align*}
I_2
=
A\ind{s}{\sigma}{\bar\delta}A_{\bar{s}\sigma\gamma}h^{\bar\delta\gamma}
=
\abs{\bar\partial{v_{\rho_\ve}}}_{\rho_\ve}^2.
\end{align*}
And the term $I_3$ is equal to $h_{s\bar\beta;\alpha}h_{\bar{s}\gamma;\bar\delta}h^{\bar\beta\gamma}h^{\bar\delta\alpha}$, which is appeared in \eqref{E:first_term}. So these terms are cancelled in the last computation.

Before computing $I_1$ and $I_4$, we introduce some ingredients. Let $R\ind{}{\delta}{\alpha\bar\beta\gamma}$ be a Riemann curvature tensor of $\rho_\ve\vert_{X_y}$. Then by the commutation formula for covariants derivatives, we have
\begin{equation} \label{E:commutation}
T\ind{}{\alpha}{;\bar\beta\gamma}
	-T\ind{}{\alpha}{;\gamma\bar\beta}
	=R\ind{}{\alpha}{\delta\bar\beta\gamma}T^\delta.
\end{equation}
Let $R_{\alpha\bar\beta}:=R\ind{}{\gamma}{\alpha\bar\beta\gamma}$ be the Ricci tensor of $\rho_\ve\vert_{X_y}$. By the definition of $h^{\rho_\ve}_{X/Y}$ in Remark 2.3, we have
\begin{equation*}
\Theta_{h^{\rho_\ve}_{X/Y}}\vert_{X_y}=-\mathrm{Ric}(\rho_\ve\vert_{X_y}).
\end{equation*}
Hence it follows from \eqref{E:Ricci} that
\begin{equation*}
R_{\alpha\bar\beta}
=
\ve h_{\alpha\bar\beta}
-
\ve g_{\alpha\bar\beta}.
\end{equation*}

\begin{lemma} \label{L:harmonic}
Let $\bar\partial^*_{\rho_\ve}$ be the adjoint of $\bar\partial$ with respect to the $L^2$-inner product with  $\rho_\varepsilon\vert_{X_y}$, which is defined by
\begin{equation*}
\bar\partial^*\paren{A\ind{s}{\alpha}{\bar\beta}
\pd{}{z^\alpha}\otimes{dz}^{\bar\beta}}
	:=h^{\bar\beta\gamma}A\ind{s}{\alpha}{\bar\beta;\gamma}\pd{}{z^\alpha}
\end{equation*}
Then we have the following:
\begin{equation}\label{E:dbarstar}
\bar\partial^*
\paren{
\bar\partial v_{\rho^\varepsilon}
}
=
\varepsilon
\paren{
	g_{s\bar\delta}h^{\bar\delta\alpha}
	-h_{s\bar\delta}g^{\bar\delta\alpha}
}
\pd{}{z^\alpha}.
\end{equation}
In particular, we have
\begin{equation*}
h^{\bar\beta\gamma}A\ind{s}{\alpha}{\bar\beta;\gamma}
=
\varepsilon
\paren{
	g_{s\bar\delta}h^{\bar\delta\alpha}
	-h_{s\bar\delta}g^{\bar\delta\alpha}
}.
\end{equation*}
\end{lemma}

\begin{proof}
Since a K\"{a}hler metric is torsion-free, we have
\begin{align*}
h^{\bar\beta\gamma}A\ind{s}{\alpha}{\bar\beta;\gamma}
	= -h^{\bar\beta\gamma}h^{\bar\delta\alpha}h_{s\bar\delta;\bar\beta\gamma}
	= -h^{\bar\beta\gamma}h^{\bar\delta\alpha}h_{s\bar\beta;\bar\delta\gamma}.
\end{align*}
By \eqref{E:Ricci} and \eqref{E:commutation}, it follows that
\begin{align*}
h^{\bar\beta\gamma}A\ind{s}{\alpha}{\bar\beta;\gamma}
&=
	 -h^{\bar\beta\gamma}h^{\bar\delta\alpha}
	\bparen{h_{s\bar\beta;\gamma\bar\delta}
	-h_{s\bar\tau}R\ind{}{\bar\tau}{\bar\beta\bar\delta\gamma}
	}
=
	-h^{\bar\delta\alpha}
	\bparen{
	\paren{h^{\bar\beta\gamma}\pd{h_{\bar\beta\gamma}}{s}}_{;\bar\delta}
	-h_{s\bar\tau}h^{\bar\beta\gamma}R\ind{}{\bar\tau}{\bar\beta\bar\delta\gamma}
	}
	\\
&=
	-h^{\bar\delta\alpha}
	\bparen{
	\paren{\pd{}{s}\log\det(h_{\alpha\bar\beta})}_{;\bar\delta}
	+h_{s\bar\tau}R\ind{}{\bar\tau}{\bar\delta}
	}
=
	-h^{\bar\delta\alpha}
	\bparen{
	(\Theta_{h^{\rho_\ve}_{X/Y}})_{s\bar\delta}
	+h_{s\bar\tau}h^{\bar\tau\gamma}R_{\gamma\bar\delta}
	}
	\\
&=
	-h^{\bar\delta\alpha}
	\bparen{
	(\Theta_{h^{\rho_\ve}_{X/Y}})_{s\bar\delta}
	-h_{s\bar\tau}h^{\bar\tau\gamma}(\Theta_{h^{\rho_\ve}_{X/Y}})_{\gamma\bar\delta}
	}
	\\
&=
	-\varepsilon
	h^{\bar\delta\alpha}
	\bparen{
	h_{s\bar\delta}-g_{s\bar\delta}
	-h_{s\bar\tau}h^{\bar\tau\gamma}
	\paren{h_{\gamma\bar\delta}-g_{\gamma\bar\delta}}
	}
=
	\varepsilon
	\paren{
		g_{s\bar\delta}h^{\bar\delta\alpha}
		-h_{s\bar\delta}g^{\bar\delta\alpha}
	}
\end{align*}
This completes the proof.
\end{proof}

Next we compute the term $I_1$:
\begin{align*}
I_1
&=
	h_{\bar{s}\sigma}h\ind{s}{\sigma}{;\gamma\bar\delta}h^{\bar\delta\gamma}
=
	h_{\bar{s}\sigma}
	\paren{
		-A\ind{s}{\sigma}{\bar\delta;\gamma}h^{\bar\delta\gamma}
		+h\ind{s}{\lambda}{}R\ind{}{\sigma}{\lambda\gamma\bar\delta}
		h^{\bar\delta\gamma}
	}
	\\
&=
	h_{\bar{s}\sigma}
	\bparen{
		-\varepsilon
		\paren{
			g_{s\bar\delta}h^{\bar\delta\sigma}
			-h_{s\bar\delta}g^{\bar\delta\sigma}
		}
		-h_{s\bar\lambda}R^{\sigma\bar\lambda}
	}
	\\
&=
	h_{\bar{s}\sigma}
	\bparen{
		-\varepsilon
		\paren{
			g_{s\bar\delta}h^{\bar\delta\sigma}
			-h_{s\bar\delta}g^{\bar\delta\sigma}
		}
		+h_{s\bar\lambda}\varepsilon
		\paren{h^{\sigma\bar\lambda}-g^{\sigma\bar\lambda}
		}
	}
	\\
&=
	\varepsilon\paren{
	h_{s\bar\beta}h^{\bar\beta\alpha}h_{\alpha\bar s}
	-
	g_{s\bar\beta}h^{\bar\beta\alpha}h_{\alpha\bar s}
	}.
\end{align*}
Finally we compute the term $I_4$:
\begin{align*}
I_4
&=
	h^{\gamma\bar\delta}h\ind{s}{\sigma}{}A_{\bar{s}\sigma\gamma;\bar\delta}
	=
	h_{s\bar\sigma}
	h^{\gamma\bar\delta}A\ind{\bar{s}}{\bar\sigma}{\gamma;\bar\delta}
	\\
&=
	h_{s\bar\sigma}
	\varepsilon\paren{
		g_{\bar s\delta}h^{\delta\bar\sigma}
		-h_{\bar s\delta}g^{\delta\bar\sigma}
	}
	=
	\varepsilon\paren{
	h_{s\bar\beta}h^{\bar\beta\alpha}g_{\alpha\bar s}
	-
	h_{s\bar\beta}g^{\bar\beta\alpha}h_{\alpha\bar s}
	}.
\end{align*}
Together with all computations, it follows that
\begin{align*}
\Delta_{\rho_\ve} c(\rho_\ve)
&=
\ve(h_{s\bar s}-g_{s\bar s})
+
\Theta_{s\bar s}
-
\abs{\bar\partial{v_{\rho_\ve}}}_{\rho_\ve}^2
%\\
%&\;\;\;
	-\ve
	\paren{
	h_{s\bar\beta}h^{\bar\beta\alpha}h_{\alpha\bar s}
	-
	g_{s\bar\beta}h^{\bar\beta\alpha}h_{\alpha\bar s}
	}
	\\
&\;\;\;-
	\ve
	\paren{
	h_{s\bar\beta}h^{\bar\beta\alpha}g_{\alpha\bar s}
	-
	h_{s\bar\beta}g^{\bar\beta\alpha}h_{\alpha\bar s}
	}
	\\
&=
	\Theta_{s\bar s}
	-
	\abs{\bar\partial{v_{\rho_\ve}}}_{\rho_\ve}^2
	+
	\ve	
	\paren{
		h_{s\bar s}
		-h_{s\bar\beta}h^{\bar\beta\alpha}h_{\alpha\bar s}
	}
	\\
&\;\;\;+	
	\varepsilon\paren{
		g_{s\bar s}
		-g_{s\bar\beta}h^{\bar\beta\alpha}h_{\alpha\bar s}
		-h_{s\bar\beta}h^{\bar\beta\alpha}g_{\alpha\bar s}
		+h_{s\bar\beta}g^{\bar\beta\alpha}h_{\alpha\bar s}
	}.
\end{align*}
Since
\begin{equation*}
\omega(v_{\rho_\ve},\overline{v_{\rho_\ve}})
=
g_{s\bar s}-g_{s\bar\beta}h^{\bar\beta\alpha}h_{\alpha\bar s}
	-h_{s\bar\beta}h^{\bar\beta\alpha}g_{\alpha\bar s}
	+h_{s\bar\beta}g^{\bar\beta\alpha}h_{\alpha\bar s},
\end{equation*}
it follows that
\begin{equation*}
-\Delta_{\rho_\ve} c(\rho_\varepsilon)
+\ve c(\rho_\varepsilon)
=
\ve\omega(v_{\rho_\ve},\overline{v_{\rho_\ve}})
+
\abs{\bar\partial{v_{\rho_\ve}}}_{\rho_\ve}^2
-
\Theta_{s\bar s}.
\end{equation*}
Therefore, we have the conclusion.
\end{proof}

\begin{corollary}\label{C:PDE}
Let $\rho$ be the fiberwise Ricci-flat metric in Theorem \ref{T:main_theorem}. Then the following PDE holds on each fiber $X_y$:
\begin{equation*}
-\Delta_\rho c(\rho)=\abs{\bar\partial v_\rho}_\rho^2-\Theta_{s\bar s}.
\end{equation*}
\end{corollary}

\begin{proof}
If we apply the same computation with the proof of Theorem \ref{T:PDE} to $\rho$ using the above equation, then we have the conclusion.

On the other hand, it is also an easy consequence of the convergence of the form $\rho_\ve$ to $\rho$ as $\ve\rightarrow0$ by passing through a subsequence for each $y\in Y$.
(More precisely, the function $\vp_\ve$ converges to $\vp$ as $\ve\rightarrow0$.) This will be proved in the next section.
\end{proof}

\begin{remark}\label{R:PDE}
The computations in Corollary \ref{C:PDE} do not use the normalization condition of $\vp$.
Hence it is easy to see that for any $d$-closed smooth real $(1,1)$-form $\tau$ whose restriction on each fiber is the Ricci-flat metric we have
\begin{equation*}
-\Delta_\tau c(\tau)
=
\abs{\bar\partial v_\tau}_\tau^2-\Theta_{s\bar s}.
\end{equation*}
\end{remark}

Now we are at the position of proving the positivity of the direct image $p_*\rho^{n+1}$.
As we mentioned in Subsection \ref{SS:horizontal_lift}, it is enough to show that the fiber integral $\int_{X_s}c(\rho)\rho^n$ is positive. 
It follows from Theorem \ref{T:PDE} that
\begin{align*}
\int_{X/\DD}
c(\rho_\ve)\rho_\ve^n
&=
\int_{X_s}
\frac{1}{\ve}
\bparen{
	\Delta_{\rho_\ve}c(\rho_\ve)
	+
	\abs{\bar\partial v_{\rho_\ve}}^2_{\rho_\ve}
	-
	\Theta_{s\bar s}
	+\ve\omega(v_{\rho_\ve},\bar v_{\rho_\ve})
}
\rho_\ve^n
\\
&=
\frac{1}{\ve}
\bparen{	
	\norm{\bar\partial v_{\rho_\ve}}^2_{L^2_{\rho_\ve}(X_s)}
	-
	\Theta_{s\bar s}
}
+
\int_{X_s}
\omega(v_{\rho_\ve},\bar v_{\rho_\ve})
\rho_\ve^n.
\end{align*}
Let $u$ be a non-vanishing $(n,0)$ form on $X_s$ satisfying the condition in Lemma \ref{L:nonvanishingform}. 
Then Proposition \ref{P:harmonic_representative} says that $\bar\partial v_\rho\cdot u$ is the harmonic representative of $K_s\cdot u$.
Hence Theorem \ref{T:Griffiths} implies that
\begin{align*}
\norm{\bar\partial v_{\rho_\ve}}_{\rho_\ve}^2
=
\norm{\bar\partial v_{\rho_\ve}\cdot u}_{\rho_\ve}^2
\ge
\norm{K_s\cdot u}^2
=
\norm{\bar\partial v_\rho\cdot u}_\rho^2
=
\norm{\bar\partial v_\rho}_\rho^2
=
\Theta_{s\bar s}.
\end{align*}
We already know that on each fiber $X_s$, the $\rho_\ve\vert_{X_s}$ converges to $\rho\vert_{X_s}$ by Corollary \ref{C:convergence_vp}.
Therefore, Proposition \ref{P:conv_geo_curv}, which will be proved in the next section, says that
\begin{equation}\label{E:lower_bound}
\int_{X/\DD}
c(\rho)\rho^n
\ge
\int_{X_s}
\omega(v_{\rho},\bar v_{\rho})
\rho^n.
\end{equation}
In particular, $p_*\rho^{n+1}$ is positive.

\begin{proposition} \label{P:conv_geo_curv}
On each fiber $X_y$, we have
\begin{equation*}
c(\rho_{\ve})\rightarrow c(\rho)
\;\;\;\text{and}\;\;\;
\bar\partial v_{\rho_{\ve}}
\rightarrow
\bar\partial v_{\rho}
\;\;\;\text{as}\;\;\;
\ve\rightarrow\infty.
\end{equation*}
\end{proposition}

\section{Approximation of the geodesic curvature}\label{S:app_geo_curv}

In this section, we shall prove Proposition \ref{P:conv_geo_curv}. 

First we recall the setting:
Let $p:X\rightarrow\DD$ be a Calabi-Yau fibration and let $\omega$ be a fixed K\"ahler form on $X$. For each fiber $X_y$, we have a unique solution $\vp_{y,\ve}$ of the following complex Monge-Amp\`ere equation:
\begin{equation}\label{E:CMAEvpve}
\begin{aligned}
\paren{\omega_y+dd^c\vp_{y,\ve}}^n
&=
e^{\varepsilon\vp_{y,\ve}}e^{\eta_y}(\omega_y)^n
\;\;\text{and}
\\
\omega_y+dd^c\vp_{y,\ve}
&>0,
\end{aligned}
\end{equation}
where $\eta$ is defined in Section \ref{S:fiberwiseRFf}.
As we mentioned, we can consider $\vp_\ve$ as a smooth function on $X$ by letting
\begin{equation*}
\vp_\ve(x):=\vp_{y,\ve}(x),
\end{equation*}
where $y=p(x)$. 
Denote by $\rho_\ve=\omega+dd^c\vp_\ve$.

On the other hand, for each fiber $X_y$, we have the solution $\vp_y$ of the following complex Monge-Amp\`ere equation:
\begin{equation}\label{E:CMAEvp}
\begin{aligned}
\paren{\omega_y+dd^c\varphi_y}^n &= e^{\eta\vert_{X_y}}(\omega_y)^n, \\
\omega_y+dd^c&\varphi_y>0,
\end{aligned}
\end{equation}
which is normalized by 
\begin{equation}\label{E:normalization}
\int_{X_y}\varphi_y e^{\eta_y}(\omega_y)^n=0.
\end{equation}
Then $\vp$ is a smooth function on $X$. We denote by $\rho=\omega+dd^c\vp$. 
It is remarkable to note that $\rho_\ve$ and $\rho$ are uniformly equivalent on $X_y$ by Proposition \ref{P:approximation1}.

In this section, we write the horizontal lifting $v_\rho$ of $\partial/\partial s$ with respect to $\rho$ as follows:
\begin{equation*}
v_\rho=\pd{}{s}+a\ind{s}{\alpha}{}\pd{}{z^\gamma}
=
\pd{}{s}-h_{s\bar\beta}h^{\bar\beta\alpha}\pd{}{z^\gamma}.
\end{equation*}
in an admissible coordinate $(z,s)$ in $X$.

\begin{theorem}\label{T:convergence}
For a fixed fiber $X_y$, the following holds:
\begin{equation*}
\vp_\ve\rightarrow \vp,
\;\;\;
v_\rho\vp_\ve\rightarrow 
v_\rho\vp
\;\;\;
\text{and}
\;\;\;
\overline{v_\rho} v_\rho\vp_\ve
\rightarrow
\overline{v_\rho}v_\rho\vp
\end{equation*}
as $\ve\rightarrow0$ in $C^{k,\alpha}(X_y)$-topology for any $k\in\NN$ and $\alpha\in(0,1)$.
\end{theorem}
It is obvious that this theorem implies Proposition \ref{P:conv_geo_curv}. 
\medskip

In the proof, we fix a fiber $X_y$ and omit the subscript $y$, if there is no confusion. It is easy to see that Corollary 3.5 yields the first assertion. 
This also implies that there exists a uniform constant $C>0$ such that
\begin{equation}\label{E:equivalence'}
\frac{1}{C}\omega_y
<
\rho_\ve\vert_{X_y}
<
C\omega_y,
\end{equation}
for $\ve>0$.
\medskip

Before going to the further proof of Theorem \ref{T:convergence}, we introduce the following proposition about the fiber integral.

\begin{proposition}\label{P:Lie_derivative}
Let $\tau$ be a real $(1,1)$-form on $X$ whose restriction on each fiber $X_s$ is positive definite.
For a smooth function $f$ on $X$, we have
\begin{equation*}
\pd{}{s}\int_{X_s}f\tau^n
=
\int_{X_s}L_{v_\tau}\paren{f\tau^n}
=
\int_{X_s}(v_\tau f)\tau^n.
\end{equation*}
In particular, if $\int_{X_s}f\tau^n=0$ for $s\in\DD$, then 
$$
\int_{X_s}(v_\tau f)\tau^n=0.
$$
\end{proposition}

\begin{proof}
The first equality is mentioned in Section 3.2. Cartan's magic formula and Stokes' theorem imply that
\begin{align*}
\pd{}{s}\int_{X_s}f\tau^n
&=
\int_{X_s}L_{v_\tau}\paren{f\tau^n}
=
\int_{X_s} \paren{d\circ i_{v_\tau}+i_{v_\tau}\circ d}
\paren{f\tau^n}\\
&=
\int_{X_s} d\paren{i_{v_\tau}\paren{f\tau^n}}
+
\int_{X_s} i_{v_\tau}\paren{df\wedge\tau^n}\\
&=
\int_{X_s} (v_\tau f)\tau^n
-
\int_{X_s} df\wedge i_{v_\tau}(\tau^n).
\end{align*}
On the other hand, Lemma \ref{L:contraction} implies that
\begin{equation*}
i_{v_\tau}(\tau^n)
=
i_{v_\tau}(\tau)\wedge\tau^{n-1}
=
\im c(\tau)\wedge\tau^{n-1}\wedge d\bar s.
\end{equation*}
Hence we have
\begin{equation*}
\int_{X_s} df\wedge i_{v_\tau}(\tau^n)
=
\int_{X_s} \im c(\tau)df\wedge\tau^{n-1}\wedge d\bar s=0.
\end{equation*}
This completes the proof.
\end{proof}

Now we go back to the proof of the second assertion. Taking logarithm of \eqref{E:CMAEvpve} and differentiating it with respect to $v_\rho$, we have
\begin{equation*}
(h_\ve)^{\bar\beta\alpha}
v_\rho\paren{g_{\alpha\bar\beta}+(\vp_\ve)_{\alpha\bar\beta}}
=
\ve v_\rho\vp_\ve
+
v_\rho\eta
+
g^{\bar\beta\alpha}v_\rho(g_{\alpha\bar\beta}).
\end{equation*}
As in Section 3, we have
\begin{equation*}
-\Delta_{\rho_\ve}\paren{v_\rho\vp_\ve}
+\ve\paren{v_\rho\vp_\ve}
=
-v_\rho\eta
+
(h_\ve)^{\alpha\bar\beta}
\paren{
v_\rho\paren{g_{\alpha\bar\beta}}
+
[v_\rho,\vp_\ve]_{\alpha\bar\beta}
}
-
g^{\alpha\bar\beta}
v_\rho\paren{g_{\alpha\bar\beta}},
\end{equation*}
where $\Delta_{\rho_\ve}$ is the Laplace-Beltrami operator of $\rho_\ve$ and
\begin{align*}
[v_\rho,\vp_\ve]_{\alpha\bar\beta}
&=
v_\rho((\vp_\ve)_{\alpha\bar\beta})-(v_\rho(\vp_\ve))_{\alpha\bar\beta} \\
&=
-a\ind{s}{\gamma}{\alpha\bar\beta}(\vp_\ve)_\gamma
-a\ind{s}{\gamma}{\alpha}(\vp_\ve)_{\gamma\bar\beta}
-a\ind{s}{\gamma}{\bar\beta}(\vp_\ve)_{\alpha\gamma}.
\end{align*}
We denote the right hand side by $R_\ve$. Hence $v_\rho\vp_\ve$ satisfies the following equation:
\begin{equation} \label{E:pde1}
-\Delta_{\rho_\ve}(v_\rho\vp_\ve)
+
\varepsilon(v_\rho\vp_\ve)
=
R_\ve.
\end{equation}
Then Proposition \ref{P:Key_Prop} implies that there exists a uniform constant $C>0$ such that
\begin{equation*}
\norm{v_\rho\vp_\ve}_{C^{k,\alpha}(X_s)}<C.
\end{equation*}

By the same computation to \eqref{E:CMAEvp}, $v_\rho\vp$ satisfies that 
\begin{equation}\label{E:pde1'}
-\Delta_\rho(v_\rho\vp)
=
R,
\end{equation}
where 
$$
R=
-v_\rho\eta
+
h^{\alpha\bar\beta}
\paren{
v_\rho\paren{g_{\alpha\bar\beta}}
+
[v_\rho,\vp]_{\alpha\bar\beta}
}
-
g^{\alpha\bar\beta}
v_\rho\paren{g_{\alpha\bar\beta}}.
$$
Since $\vp_\ve$ converges to $\vp$ and
$[v_\rho,\vp_\ve]_{\alpha\bar\beta}$ does not include $s$-derivative of $\vp_\ve$, we have
\begin{equation*}
(h_\ve)^{\bar\beta\alpha}\rightarrow
h^{\bar\beta\alpha}
\;\;\;\text{and}\;\;\;
[v_\rho,\vp_\ve]_{\alpha\bar\beta}
\rightarrow
[v_\rho,\vp]_{\alpha\bar\beta}
\;\;\;\text{as}\;\;\;
\ve\rightarrow0.
\end{equation*}
It follows that Equation \eqref{E:pde1} converges to Equation \eqref{E:pde1'} as $\ve\rightarrow0$. 
Since Proposition \ref{P:Lie_derivative} says that $v_\rho\vp$ is the unique solution of \eqref{E:pde1'} which satisfies that
\begin{equation*}
\int_{X_s}(v_\rho\vp)\rho^n=0,
\end{equation*}
the following Lemma completes the proof.

\begin{lemma}\label{L:initial1}
The following holds:
\begin{equation*}
\lim_{\ve\rightarrow0}
\int_{X_s}(v_\rho\vp_\ve)\rho^n=0.
\end{equation*}
\end{lemma}

\begin{proof}
Integrating \eqref{E:CMAEvpve}, we have
\begin{equation}\label{E:totalmassve}
1=\int_{X_s}e^{\ve\vp_\ve+\eta}\omega^n.
\end{equation}
Differentiating with respect to $s$, we have
\begin{equation*}
0
=
\pd{}{s}\int_{X_s}e^{\ve\vp_\ve+\eta}\omega^n
=
\int_{X_s}v_\rho(e^{\ve\vp_\ve})\rho^n
=
\ve\int_{X_s}(v_\rho\vp_\ve)e^{\ve\vp_\ve}\rho^n.
\end{equation*}
Since $e^{\ve\vp_\ve}(\rho_\ve)^n=\rho^n$ on each fiber $X_s$,
\begin{equation*}
\int_{X_s}(v_\rho\vp_\ve)(\rho_\ve)^n=0.
\end{equation*}
Since $\rho_\ve$ and $\rho$ is uniformly equivalent on $X_s$, this completes the proof.
\end{proof}

It remains only to prove the last assertion. 
\medskip

Differentiating \eqref{E:pde1} with respect to $\overline{v_\rho}$, we have
\begin{equation}\label{E:PDEss1}
\begin{aligned}
-\Delta_{\rho_\ve}(\overline{v_\rho}v_\rho\vp_\ve)
+
\varepsilon(\overline{v_\rho}v_\rho\vp_\ve)
=&
\overline{v_\rho}\paren{(h^\ve)^{\bar\beta\alpha}}\cdot(v_\rho(\vp_\ve))_{\alpha\bar\beta}
+
\overline{v_\rho}(R_\ve)\\
&+
(h_\ve)^{\bar\beta\alpha}[\overline{v_\rho},v_\rho\vp_\ve]_{\alpha\bar\beta}.
\end{aligned}
\end{equation}
Then Proposition \ref{P:Key_Prop} implies that there exists a uniform constant $C>0$ such that
\begin{equation*}
\norm{\overline{v_\rho}v_\rho\vp_\ve}_{C^{k,\alpha}(X_s)}<C.
\end{equation*}
By the same way, $\overline{v_\rho}v_\rho\vp$ satisfies that
\begin{equation}\label{E:PDEss0}
-\Delta_\rho\overline{v_\rho}v_\rho\vp
=
\overline{v_\rho}\paren{h^{\bar\beta\alpha}}\cdot(v_\rho\vp)_{\alpha\bar\beta}
+
\overline{v_\rho}R
+
h^{\bar\beta\alpha}[\overline{v_\rho},v_\rho\vp]_{\alpha\bar\beta}.
\end{equation}
We already know that $\vp_\ve\rightarrow\vp$ and $v_\rho\vp_\ve\rightarrow v_\rho\vp$ as $\ve\rightarrow0$ on $X_y$. 
Hence the similar argument says that the RHS of \eqref{E:PDEss1} converges to the RHS of \eqref{E:PDEss0} as $\ve\rightarrow0$.
Since Proposition \ref{P:Lie_derivative} says that $\overline{v_\rho}v_\rho\vp$ is the unique solution of \eqref{E:PDEss0} which satisfies that
\begin{equation*}
\int_{X_s}(\overline{v_\rho}v_\rho\vp)\rho^n=0,
\end{equation*}
As the previous argument, the following lemma completes the proof.

\begin{lemma}
The following holds:
\begin{equation*}
\lim_{\ve\rightarrow0}
\int_{X_s}(\overline{v_\rho}v_\rho\vp_\ve)\rho^n=0.
\end{equation*}
\end{lemma}

\begin{proof}
Differentiating \eqref{E:totalmassve} with respect to $s$ and $\bar s$, we have
\begin{align*}
0
&=
\pd{^2}{\bar s\partial s}\int_{X_s}e^{\ve\vp_\ve+\eta}\omega^n
=
\int_{X_s}(\overline{v_\rho}v_\rho e^{\ve\vp_\ve})\rho^n\\
&=
\ve\int_{X_s}(\overline{v_\rho}v_\rho\vp_\ve)e^{\ve\vp_\ve}\rho^n
+
\ve^2\int_{X_s}\abs{v_\rho\vp_\ve}^2e^{\ve\vp_\ve}\rho^n.
\end{align*}
Since $\vp_\ve$ and $v_\rho\vp_\ve$ are uniformly bounded, it follows that
\begin{equation*}
\int_{X_s}(\overline{v_\rho}v_\rho\vp_\ve)(\rho_\ve)^n
\rightarrow0
\;\;\;
\text{as}
\;\;\;
\ve\rightarrow0.
\end{equation*}
This completes the proof as in the proof of Lemma \ref{L:initial1}.
\end{proof}

\noindent{\bf Acknowlegement.} 
The second author happily acknowledges his thanks to Mihai P\v aun who suggested this problem, shared his ideas. He is also indebted to Hoang Lu Chinh, Bo Berndtsson, Henri Guenancia, Long Li for many helpful comments and discussions. Finally, he would like to thank Y. Ge for informing the error of the previous version to the author.

The second author was supported by the National Research Foundation(NRF) of Korea grant funded by the Korea government (No. 2018R1C1B3005963).

\end{document}